\documentclass[11pt]{llncs}
\usepackage[margin=1in]{geometry}

\usepackage{graphicx,amsmath,amsfonts,bm}
\usepackage[toc,page]{appendix}
\usepackage{textcomp}
\usepackage{float}
\usepackage{mathtools}
\usepackage[affil-it]{authblk}
\usepackage{amsmath,mathrsfs}
\usepackage{mwe}
\usepackage{subfig}
\usepackage{color}

\def\Lx{L_x}
\def\betanplus{\beta_n^+}
\def\betanminus{\beta_n^-}
\def\betanplusminus{\beta_n^\pm}

\newcommand\restr[2]{{% we make the whole thing an ordinary symbol
  \left.\kern-\nulldelimiterspace % automatically resize the bar with \right
  #1 % the function
  \vphantom{\big|} % pretend it's a little taller at normal size
  \right|_{#2} % this is the delimiter
  }}
\newcommand{\ou}{%
  \mathrel{%
    \vcenter{\offinterlineskip
      \ialign{##\cr$L+1$\cr\noalign{\kern-1.5pt}$1$\cr}%
    }%
  }%
}

\newcommand\restrict[3]{{% we make the whole thing an ordinary symbol
  \left.\kern-\nulldelimiterspace % automatically resize the bar with \right
  #1 % the function
  \vphantom{\big|} % pretend it's a little taller at normal size
  \right|^{#3}_{#2} % this is the delimiter
  }}
\newcommand{\normblank}[2]{\left\Vert #1\right\Vert_{#2}}
\newcommand{\normH}[2]{\left\Vert #1\right\Vert_{H^{1}(#2)}}

\newcommand{\normL}[2]{\left\Vert #1\right\Vert_{L^{2}(#2)}}
\newcommand{\normLp}[3]{\lVert #1\rVert_{L^{#3}(#2)}}
\newcommand{\normHs}[2]{\lVert #1\rVert_{H^{1}(#2)}^{2}}

\newcommand{\normLs}[2]{\left\Vert #1\right\Vert_{L^{2}(#2)}^{2}}
\newcommand{\normLinf}[2]{\left\Vert #1\right\Vert_{L^{\infty}(#2)}}

\newcommand{\normHonesone}[2]{\left\Vert #1\right\Vert_{H^{1+s_{1}}(#2)}}

\DeclareMathOperator*{\essinf}{ess\,inf}
\begin{document}

\title{Analysis of the Rigorous Coupled Wave Approach for $p$-polarized light in gratings\thanks{Supported by the US National Science Foundation (NSF) under grant numbers DMS-1619904 and DMS-1619901.}}

\author{Benjamin J. Civiletti\inst{1}, Akhlesh Lakhtakia
\inst{2} \and
Peter B. Monk\inst{1} }
\authorrunning{B.J. Civiletti et al.}

\institute{Department of Mathematical Sciences, University of Delaware, Newark, DE, United States
of America \and  Department of
Engineering Science and Mechanics, Pennsylvania State University, 
University Park, PA, United States
of America}

\maketitle              % typeset the header of the contribution
\begin{abstract}
We study the convergence properties of the two-dimensional Rigorous Coupled Wave Approach (RCWA)   for 
$p$-polarized monochromatic incident light. The RCWA is a semi-analytical numerical method that is widely used to solve  the boundary-value problem of scattering  by a grating. The approach requires the expansion of all  electromagnetic field phasors and the relative permittivity as Fourier series in the spatial variable
along the direction of the  periodicity of the grating. In the direction perpendicular to the grating periodicity, the domain is discretized into thin slices and the actual relative permittivity is replaced by an approximation. The approximate relative permittivity is chosen so that the solution of the Maxwell equations in each slice can be computed without further approximation. Thus, there is error due to the
approximate relative permittivity as well as the trucation of the Fourier series. We show that the RCWA  embodies a Galerkin scheme for a perturbed problem, and then we use tools from the Finite Element Method to show that the method converges with increasing number of retained Fourier modes and finer approximations of the relative permittivity.  Numerical examples illustrate our analysis, and suggest further work.

\end{abstract}
\section{Introduction}

The Rigorous Coupled Wave Approach (RCWA) is a popular numerical method for solving electromagnetic  scattering problems involving periodic structures \cite{Maystre,Popov,Gaylord,Faryad,Akhlesh2}. Floquet theory \cite{Yakubovich} shows that the true solutions to these problems are quasi-periodic in the same direction as the periodicity of the grating and can be represented using Fourier series in that same direction. The RCWA exploits this fact by expanding the electric and magnetic field phasors and the constitutive parameters such as
the relative permittivity as Fourier series along the direction of periodicity. After substitution of these representations into Maxwell's equations and truncation of the Fourier series for computational tractability, a system of first-order ordinary differential equations (ODEs) is obtained \cite{Faryad,Akhlesh2}.  We shall prove the convergence of the  RCWA applied to a grating that is translationally invariant in one direction.

The case we study in this paper is the one in which the domain is illuminated by $p$-polarized monochromatic incident light. Then the approach is equivalent to solving a system of second-order ODEs relating the Fourier modes of the magnetic field phasor \cite{Li}. However, even this truncated system is difficult to solve, so the true relative permittivity is replaced by an approximation whereby the domain is discretized into thin slices in the direction perpendicular to the periodicity of the grating. In each slice, the true relative permittivity is approximated so that the solution in each slice can be computed exactly. A fast linear-algebraic algorithm can then be derived for this problem, by enforcing the continuity of the Fourier modes and its derivative across the inter-slice boundaries \cite{Gaylord3,Lalanne1996,Akhlesh2}. The approximate solution in the entire domain is then formed by stitching together the solutions for the slices.

The basic idea of the RCWA  derives from coupled-wave analysis for diffraction problems \cite{coupled}. It was proposed in its current form in \cite{Gaylord}. %There is no numerical instability encountered for $s$-polarized incident light,   but that is not true for $p$-polarized incident light. 
The convergence of the field phasors in the vicinity of the grating in the $p$-polarization case was improved greatly by Li \cite{Li} by slightly reformulating the system of second-order ODEs. The RCWA  is now  standard  for rapidly computing the near field in grating problems \cite{Maystre,Popov}, and it is
particularly attractive for designing optimal thin-film solar cells \cite{Solcore,S4,Ahmad-CIGS,Ahmad-AlGaAs,Tom}.

Our contribution in this paper is to provide the first convergence proof for the RCWA   for $p$-polarized incident light. This proof involves extending well-posedness results  for standard scattering problems \cite{Graham} to scattering in a periodic medium, followed by a perturbation and variational analysis of the problem. Our analysis covers the case where the relative permittivity $\varepsilon$ is everywhere real and the device is non-trapping and also when $\varepsilon$ is everywhere complex, the latter case being more relevant for solar cells. This paper follows the general approach  of our analysis of
 $s$-polarized
incident light \cite{CivilettiMonk}, but   the proof for $p$-polarized incident light is more difficult since the field phasors vary less smoothly, and the perturbed coefficient appears in the principle part of the differential operator governing
the field phasors.

This paper is organized as follows. In Section~\ref{formulation}, we introduce the mathematical problem: an inhomogeneous Helmholtz equation with quasi-periodic boundary conditions. We  give the variational formulation for our problem in Section~\ref{variational}. In Section \ref{Rellich} we derive a Rellich identity for solutions of the Helmholtz problem, assuming that $\varepsilon$ is real and $C^{\infty}$ smooth. Then we show that an \textit{a-priori} estimate holds, where the continuity constant has explicit dependence on $\varepsilon$ as long as certain non-trapping conditions are fulfilled. These non-trapping conditions ensure, roughly speaking, that a quantum of electromagnetic energy entering the domain leaves it after a finite time. We extend these \textit{a-priori} estimates to more general $\varepsilon$ in Section \ref{Linf}, using generalized non-trapping conditions \cite{Graham}. In particular, we use these results to show that if $\varepsilon$ is piecewise smooth and satisfies the generalized non-trapping conditions, then the perturbed $\varepsilon$ problem also satisfies them. Convergence in slice thickness is shown in Section \ref{thickness} to follow from the foregoing conclusion. We further extend the \textit{a-priori} estimates obtained in Section \ref{norm} to demonstrate in Section~\ref{Galerkin} that the RCWA   embodies a Galerkin scheme. Convergence  in retained Fourier modes is shown in Section \ref{modes}. The case where $\varepsilon$ is everywhere complex is considered in Section \ref{absorbing}. Finally, we test our prediction of convergence order with some numerical examples in Section \ref{numerics} by comparing the RCWA solution with the solution delivered by a highly refined Finite Element Method.

\section{Formulation of the continuous problem}
\label{formulation}
We consider  linear optics with an $\exp(-i \omega t)$ dependence on time $t$, where $i=\sqrt{-1}$ and $\omega$ is the angular frequency of light. The grating is assumed to be invariant in the $x_{3}$-direction.
For  $p$-polarized light,
the electric field phasor is $\bm{E}= E_1{\bf e}_1 + E_2{\bf e}_2=(E_{1},E_{2},0)$ and the magnetic field phasor is $\bm{H}= H_3{\bf e}_3=(0,0,H_{3})$, where the unit vectors $\bm{e}_1=(1,0,0)$, $\bm{e}_2=(0,1,0)$, and
$\bm{e}_3=(0,0,1)$ \cite{Guenther}. All non-zero components of the two phasors depend on $x_{1}$ and $x_{2}$ but not on $x_3$. The spatially dependent relative permittivity is denoted by $\varepsilon(x_{1},x_{2})$ whereas the relative permeability is assumed to be unity everywhere. The wavenumber in air is denoted
by $\kappa=\omega\sqrt{\varepsilon_0\mu_0}$ and the intrinsic impedance of air by $\eta=\sqrt{\mu_0/\varepsilon_0}$, where $\varepsilon_0=8.854\times 10^{-12}$~F~m$^{-1}$
is the permittivity   and $\mu_0=4\pi\times10^{-7}$~H~m$^{-1}$ is the permeability of vacuum.

 Under the above assumptions, we notice that
\begin{equation}
\nabla \times \bm{H}=\bigg( \frac{\partial H_{3}}{\partial x_{2}},-\frac{\partial H_{3}}{\partial x_{1}},0\bigg).
\end{equation}
By virtue of the Amp\`ere--Maxwell equation $-i (\kappa/\eta) \varepsilon \bm{E}=\nabla \times \bm{H}$,  it follows that 
\begin{align}
-i (\kappa/\eta) \varepsilon E_{1}&=\frac{\partial H_{3}}{\partial x_{2}},
\label{AL-eq2}\\
 i (\kappa/\eta)\varepsilon E_{2}&=\frac{\partial H_{3}}{\partial x_{1}}.
 \label{AL-eq3}
\end{align}
From the Faraday equation $i \kappa\eta \bm{H}=\nabla \times \bm{E}$, we have
\begin{equation}
H_{3}=\frac{1}{i\kappa\eta}\bigg(\frac{\partial E_{2}}{\partial x_{1}}-\frac{\partial E_{1}}{\partial x_{2}} \bigg)\,.
\label{AL-eq4}
\end{equation}
Combining (\ref{AL-eq2})--(\ref{AL-eq4}), we finally obtain
\begin{equation}
H_{3}=-\frac{1}{\kappa^{2}} \bigg[ \nabla \cdot \bigg(\frac{1}{\varepsilon}\nabla H_{3} \bigg) \bigg].
\end{equation}
 We denote the total field $H_{3}$ by $u^{t}$ in the remainder of the paper to simplify notation.

We now present the standard mathematical formulation for the chosen scattering problem. The relative permittivity $\varepsilon(x_1,x_2)$ is assumed to be periodic in the $x_{1}$ direction with period $\Lx>0$. The
two-dimensional rectangular domain containing one period is denoted by 
$\Omega=\{\bm{x} \in \mathbb{R}^{2} , \ 0<x_{1}<\Lx, -H<x_{2}<H \}$.
The quantity $H>0$ is chosen large enough so that $\varepsilon(x_{1},x_{2})=\varepsilon_{+}$ for $x_{2}>H$ and $\varepsilon(x_{1},x_{2})=\varepsilon_{-}$ for $x_{2}<-H$, both  $\varepsilon_{+}$ and $\varepsilon_{-}$ being positive constants. 

A $p$-polarized   plane wave propagating in the half-space $x_2>H$
at an angle $\theta$ with respect to the $x_2$-axis is incident on the plane
$x_2=H$; the sole non-zero component of its magnetic field phasor is denoted by
\begin{equation}
\label{incident}
u^{\text{inc}}(x_{1},x_{2})=\exp \bigg[i \kappa \sqrt{\varepsilon_{+}}\big(x_{1}\sin \theta - x_{2} \cos \theta \big) \bigg]\,.
\end{equation}
The scattered magnetic field phasor $u\bm{e}_{3}$ is given in terms of the total field by $u=u^t-u^{\rm{}inc}$, where $u$ satisfies the Helmholtz equation
\begin{align}
    \label{eqn:PDE}
    \nabla \cdot \bigg(\frac{1}{\varepsilon}\nabla u\bigg)+\kappa^{2}u &=f \hspace{2.1cm} \text{in} \ \Omega,\\
     \label{eqn:PDE1}
    \exp(-i \alpha_0 \Lx)u(0,x_{2})&=u(\Lx,x_{2}) \hspace{1cm}  \forall\, x_{2},\\
    \exp(-i \alpha_0 \Lx)\frac{\partial }{\partial x_{2}}u(0,x_{2})&=\frac{\partial }{\partial x_{2}}u(\Lx,x_{2}) \hspace{0.3cm} \forall\, x_{2},
    \label{eqn:PDE2}
\end{align}
and
\begin{equation}
\alpha_0=\kappa  \sqrt{\varepsilon_{+}} \sin \theta\,.
\end{equation}
The source function in the above Helmholtz equation (\ref{eqn:PDE}) is 
\begin{equation}
\label{AL-eq10}
f=\nabla \cdot \big[(\varepsilon_{+}^{-1}-\varepsilon^{-1})\nabla u^{\text{inc}} \big].
\end{equation} In the first part of the paper, we assume that $\varepsilon \in C^{\infty}(\mathbb{R}^{2})$ so that $f \in L^{2}(\Omega)$.  However, we are interested in the case where $\varepsilon$ is only piecewise smooth so that   $f \notin L^{2}(\Omega)$ in general. We discuss the regularity of $f$ in more detail later in the paper.  In addition, we    note that $u$ is quasi-periodic in $\Omega$, accounting for the multiplicative factors in equations (\ref{eqn:PDE1}) and (\ref{eqn:PDE2}).

We also assume that $\varepsilon$ is piecewise $C^{2}$ in $\mathbb{R}^{2}$ and that either 
\begin{enumerate}
\item[I.] $\varepsilon$ is real and $\Re(\varepsilon)>0$, or
\item[II.] $\varepsilon$ is complex, $\Im(\varepsilon)>c_{1}>0$, and  $\Re(\varepsilon)>c_{2}>0$ in $\Omega$
and a positive real constant elsewhere.
\end{enumerate}
Case~I encompasses insulators whereas Case~II covers dissipative dielectric materials (but not metals).

To show convergence of the RCWA   in Case~I, we prove a Rellich identity for the problem and show that an \textit{a-priori} estimate holds when $\varepsilon \in C^{\infty}(\mathbb{R}^{2})$. Using a technique of Graham \textit{et al.} \cite{Graham}, we then show that several different \textit{a-priori} bounds hold for the chosen problem, even if $\varepsilon \in L^\infty(\Omega)$. We assume  certain non-trapping conditions in order to ensure that the continuity constants in the \textit{a-priori} estimates can be written explicitly in terms of $\kappa$ and $\varepsilon$. In Case~II, the problem is coercive and we employ Strang lemmas \cite{Previato}  to prove convergence; hence, non-trapping conditions are unnecessary.  

\subsection{Variational formulation}
\label{variational}
To prove convergence of the RCWA, we need to consider several different variational problems, because the approach replaces the true $\varepsilon$ with an approximation $\varepsilon_{h}$. To define this approximation, the domain $\Omega$ is decomposed into thin  slices stacked along the
$x_2$-axis,  using a mesh $-H=h_{0}<h_{1}< \cdots < h_{S}=H$ for some $S >0$. The slices are specified as
\begin{equation}
S_{j}=\{\bm{x}\in \mathbb{R}^{2} , \ 0<x_{1}<\Lx, h_{j-1}<x_{2}<h_{j} \}, \quad j\in[1,S].
\end{equation}
%are stacked along the $x_{2}$-axis, 
From this the slice thickness parameter  $h=\max_{j} (h_{j}-h_{j-1})$. In the
$j$-th slice $S_{j}$, the true relative permittivity $\varepsilon$ is sampled on the center line of the slice to yield
\begin{equation}
\varepsilon_{h}(x_{1},x_{2})=\varepsilon(x_{1},h_{j-\frac{1}{2}}),\quad \bm{x}\in S_{j},
\end{equation}
where $h_{j-\frac{1}{2}}=\frac{1}{2}(h_{j}+h_{j-1})$ for each $j$.  Defined thus piecewise  in $\Omega$,   $\varepsilon_{h}$ amounts to a stairstep approximation of $\varepsilon$.

Now we define the space $H_{qp}^{1}(\Omega)$ as the completion of the quasi-periodic smooth functions $C_{qp}^{\infty}(\Omega)$ in the $H^{1}(\Omega)$ norm. By virtue of the 
trace theorem \cite{BrennerScott}, the functions in $H_{qp}^{1}(\Omega)$ are also quasi-periodic and $H_{qp}^{1}(\Omega)$ is a closed subspace of $H^{1}(\Omega)$. 

Next, let us replace the source function $f$ defined in \eqref{AL-eq10} by a more general source function
denoted by $F$.
Slightly abusing notation, we now let $u \in H_{qp}^{1}(\Omega)$ be a solution of the Helmholtz problem \eqref{eqn:PDE}--\eqref{eqn:PDE2} for $ F \in (H_{qp}^{1}(\Omega))^{'}$, the dual space of $H_{qp}^{1}(\Omega)$. Multiplying \eqref{eqn:PDE} by a test function $v \in H_{qp}^{1}(\Omega)$ and using the divergence theorem in the usual way, we obtain
\begin{eqnarray}
&&\int_{\Omega} \bigg( \frac{1}{\varepsilon}\nabla u \cdot \nabla\overline{ v}
-\kappa^{2} u \overline{v} \bigg)+\int_{\Gamma_{-H}}\overline{v} \frac{1}{\varepsilon_{-}}\frac{\partial u}{\partial x_{2}}-\int_{\Gamma_{H}}\overline{v}\frac{1}{\varepsilon_{+}}\frac{\partial u}{\partial x_{2}}=-\int_{\Omega} F \overline{v}
\label{vari14}
\end{eqnarray}
for all $v \in H_{qp}^{1}(\Omega)$, where the overbar denotes complex conjugation.  As depicted in Fig.~\ref{Fig1},
$\Gamma_{H}=\{ \bm{x}\in \mathbb{R}^{2}, \ 0<x_{1}<\Lx, x_{2}=H\}$ is the top boundary 
and
 $\Gamma_{-H}=\{ \bm{x}\in \mathbb{R}^{2} , \ 0<x_{1}<\Lx, x_{2}=-H\}$ is the  bottom
boundary of $\Omega$. In deriving (\ref{vari14}), we used the fact that
the integral on the left boundary $\Gamma_{L}=\{\bm{x}\in \mathbb{R}^{2} , \ x_{1}=0, -H<x_{2}<H \}$
of $\Omega$ is canceled out by the corresponding integral on the right boundary $\Gamma_{R}=\{\bm{x}\in \mathbb{R}^{2} , \ x_{1}=\Lx, -H<x_{2}<H \}$ of $\Omega$, since $\overline{v}\frac{1}{\varepsilon}\frac{\partial u}{\partial x_{1}}$ is periodic with period $\Lx$ along the $x_1$ direction for every $v \in H_{qp}^{1}(\Omega)$. 

%%%%%%%%%%%%%%%%%%% Figure 1 begins %%%%%%%%%%%%%%%%%
\begin{figure}[ht]
	\centering
	\includegraphics[width=2.5in,height=3.5in]{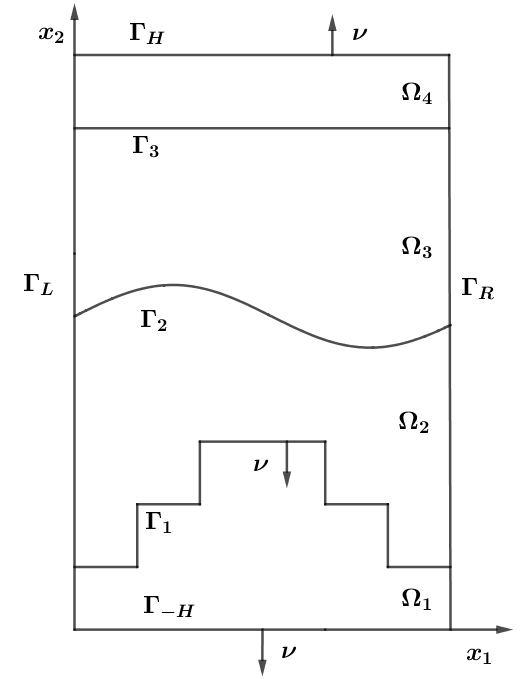}
	\caption{Geometry of the scattering problem, with $I=3$ interfaces. The domain $\Omega$ is enclosed 
	by the top boundary  $\Gamma_{H}=\{ \bm{x}\in \mathbb{R}^{2} ,\ 0<x_{1}<\Lx, x_{2}=H\}$, the right boundary
	$\Gamma_{R}=\{\bm{x}\in \mathbb{R}^{2} , \ x_{1}=\Lx, -H<x_{2}<H \}$, the bottom boundary
	$\Gamma_{-H}=\{ \bm{x}\in \mathbb{R}^{2} , \ 0<x_{1}<\Lx, x_{2}=-H\}$, and the left boundary
	$\Gamma_{L}=\{\bm{x}\in \mathbb{R}^{2} , \ x_{1}=0, -H<x_{2}<H \}$.  The relative permittivity ${\varepsilon}\in C^{2}$ in each $\Omega_{k}$, but it can jump over each interface $\Gamma_{k}$. The interface $\Gamma_{1}$ is a stairstep.  The normal vector $\nu$ is shown.}
	\label{Fig1}
\end{figure}
%%%%%%%%%%%%%%%%%%% Figure 1 ends %%%%%%%%%%%%%%%%%

At the top and bottom boundaries, we use the Dirichlet-to-Neumann   operator defined as follows. If $\phi \in H^{1/2}_{qp}(\Gamma_{H})$, then
\begin{equation}
\phi(x_{1})=\sum_{n\in \mathbb{Z}}\phi_{n} \exp(i \alpha_{n} x_{1}),
\end{equation}
where 
\begin{equation}
\alpha_{n}=\alpha_0+n\left({2 \pi }/{\Lx}\right).
\end{equation}
  In the region $\Omega^{+}=\{\bm{x}\in \mathbb{R}^{2}, 0<x_{1}<\Lx, x_{2}>H \}$, let $v_{\phi} \in H_{qp,loc}^{1}(\Omega^{+})$ satisfy
\begin{align}
\Delta v_{\phi}+\kappa^{2}\varepsilon_{+} v_{\phi}&= 0 \hspace{0.8cm}  \text{in} \ \Omega^{+}, \\
v_{\phi}&=\phi \hspace{0.8cm} \text{on} \ \Gamma_{H},
\end{align}
together with the upward propagating radiation condition \cite{Monk}. Then the Rayleigh--Bloch expansion
\begin{equation}
\label{Rayleigh}
v_{\phi}(x_{1},x_{2})=\sum_{n \in \mathbb{Z}}\phi_{n}\exp[i(x_{2}-H)\betanplus]\exp(i \alpha_{n} x_{1})
\end{equation}
follows with 
\begin{equation}
\betanplusminus= \begin{cases} 
      \sqrt{\kappa^{2}\varepsilon_{\pm}-\alpha_{n}^{2}} &\quad \alpha_{n}^{2}<\kappa^{2}\varepsilon_{\pm}, \\[5pt]
      i\sqrt{\alpha_{n}^{2}-\kappa^{2}\varepsilon_{\pm}} &\quad \alpha_{n}^{2}> \kappa^{2}\varepsilon_{\pm}.
   \end{cases}
\end{equation}
Our choice of $\betanplus$ ensures that all the modes are either   outgoing waves that do
not decay as $x_2\to\infty$ or   evanescent waves that decay as $x_{2} \to   \infty$. We define the Dirichlet-to-Neumann operators $T^{\pm}$ on the top and bottom boundaries, respectively, as
\begin{equation}
(T^{\pm} \phi)(x_{1})=\pm \frac{1}{\varepsilon_{\pm}} \restr{\frac{\partial v_{\phi}}{\partial x_{2}}}{x_{2}=\pm H}=\frac{i}{\varepsilon_{\pm}}\sum_{n \in \mathbb{Z}}\phi_{n} \beta_{n}^{\pm} \exp(i \alpha_{n}x_{1}).
\end{equation}

The  normal derivatives in \eqref{vari14} can now be replaced by Dirichlet-to-Neumann operators. 
The resulting sesquilinear form $B_{\epsilon}(\cdot,\cdot): H_{qp}^{1}(\Omega)\times H_{qp}^{1}(\Omega) \to \mathbb{C}$ is 
\begin{equation}
\label{vari}
B_{\varepsilon}(w,v)=\int_{\Omega} \bigg( \frac{1}{\varepsilon}\nabla w \cdot \nabla\overline{ v}-\kappa^{2} w \overline{v} \bigg)-\int_{\Gamma_{-H}}\overline{v} T^{-}(w)-\int_{\Gamma_{H}}\overline{v}T^{+}(w),
\end{equation}
for all $w \in H_{qp}^{1}(\Omega)$ and $v \in H_{qp}^{1}(\Omega)$.
Given an $F \in (H_{qp}^{1}(\Omega))^{'}$, we seek a solution $u \in H_{qp}^{1}(\Omega)$ such that
\begin{equation}
\label{p1}
B_{\epsilon}(u,v)=F(v)
\end{equation}
for all $v \in H_{qp}^{1}(\Omega)$, where $F(v)=-\int_{\Omega}F\overline{v}$. 

We are also interested in a perturbed problem in which $\varepsilon$ is replaced by $\varepsilon_{h}$. Therefore, we define $B_{\varepsilon_{h}}(\cdot,\cdot)$ to be the same as $B_{\varepsilon}(\cdot,\cdot)$ but with $\varepsilon_{h}$ instead of $\varepsilon$. Given an $F \in (H_{qp}^{1}(\Omega))^{'}$, we seek a solution $u^{h} \in H_{qp}^{1}(\Omega)$ such that
\begin{equation}
\label{p2}
B_{\epsilon_{h}}(u^{h},v)=F(v)
\end{equation}
for all $v \in H_{qp}^{1}(\Omega)$. 

To show that both of the foregoing problems have unique solutions, we show that either a Rellich identity holds for our problem and implies an \textit{a-priori} estimate, or the problem is coercive depending on our assumptions on $\varepsilon$. 

We now recall some properties of the Dirichlet-to-Neumann boundary integrals appearing in the sesquilinear form \eqref{vari}. The signs of the real and imaginary parts of the Dirichlet-to-Neumann  integral on ${\Gamma_{H}}$ are known because
\begin{equation}
\label{Dtn}
\left. \begin{array}{l}
\Re \int_{\Gamma_{H}}\varepsilon_{+}\overline{u} T^{+}(u)=
\displaystyle{-\sum_{\alpha_{n}^{2}>\kappa^{2}\varepsilon_{+}}\sqrt{\alpha_{n}^{2}-\kappa^{2}\varepsilon_{+}}\big| u_{n}(H)\big|^{2}}\\
[14pt]
\Im \int_{\Gamma_{H}}\varepsilon_{+}\overline{u} T^{+}(u)=\displaystyle{\sum_{\alpha_{n}^{2}<\kappa^{2}\varepsilon_{+}}\sqrt{\kappa^{2}\varepsilon_{+}-\alpha_{n}^{2}}\big| u_{n}(H)\big|^{2}}
\end{array}\right\}\,.
\end{equation}
The same signs also hold for the real and imaginary parts of the Dirichlet-to-Neumann integral on ${\Gamma_{-H}}$. These facts are used many times throughout this paper. Also, we have chosen to avoid Rayleigh--Wood anomalies \cite{Maystre,Palmer} by assuming that $\alpha_{n} \neq \kappa \sqrt{\varepsilon_{+}}$ and $\alpha_{n} \neq \kappa \sqrt{\varepsilon_{-}}$
 for any $n$.

In some arguments throughout this paper, it is useful to use the total magnetic field $u^{t}=u+u^{\text{inc}}$, instead of the scattered field $u$. Therefore, we conclude this section by giving the variational problem for the total magnetic field. We seek a $u^{t} \in H^{1}_{qp}(\Omega)$ such that
\begin{equation}
\label{tvari}
B_{\varepsilon}(u^{t},v)=\int_{\Gamma_{H}}\overline{v} \bigg(\frac{1}{\varepsilon_{+}}\frac{\partial u^{\text{inc}}}{\partial x_{2}}-T^{+}(u^{\text{inc}}) \bigg)
\end{equation} 
for all $v \in H^{1}_{qp}(\Omega)$.
\section{A Rellich identity for quasi-periodic solutions}
\label{Rellich}
In this section, we apply techniques developed by Lechleiter \& Ritterbusch \cite{Lechleiter} for scattering by an arbitrarily rough surface
to our quasi-periodic case. We assume that $\varepsilon \in C^{\infty}(\mathbb{R}^{2})$ here, and later on we will show that similar \textit{a-priori} estimates hold even when $\varepsilon \in L^{\infty}(\mathbb{R}^{2})$. These estimates will be used to address Case~I of Section~\ref{formulation}.

\begin{theorem}
\label{rellich}
Assume that $\varepsilon \in C^{\infty}(\mathbb{R}^{2})$ is real and $\Re (\varepsilon) >0$. If $u \in H_{qp}^{1}(\Omega)$ is a solution to the variational problem \eqref{p1} for a source $F \in L^{2}(\Omega)$, then the following Rellich identity holds:
\begin{align}
\nonumber
&\int_{\Omega}\bigg[\frac{2}{\varepsilon}\bigg|\frac{\partial u}{\partial x_{2}} \bigg|^{2} -(x_{2}+H)\frac{\partial}{\partial x_{2}}\bigg(\frac{1}{\varepsilon}\bigg)\big|\nabla u \big|^{2}\bigg]+2H\int_{\Gamma_{H}}\bigg(\frac{-2}{\varepsilon}\bigg|\frac{\partial u}{\partial x_{2}} \bigg|^{2}+\frac{1}{\varepsilon}\big|\nabla u \big|^{2}-\kappa^{2}\big|u \big|^{2} \bigg)\\
\nonumber
&-\int_{\Gamma_{H}}\overline{u}T^{+}(u)-\int_{\Gamma_{-H}}\overline{u}T^{-}(u)\\
&=-2\int_{\Omega}(x_{2}+H)\Re \bigg(\overline{F} \frac{\partial u}{\partial x_{2}} \bigg)-\int_{\Omega}F \overline{u}.
\label{AL-eq24}
\end{align}
\end{theorem}
\begin{proof}
The proof can be found in the Appendix.
\end{proof}

Now we use the Rellich identity to show that an \textit{a-priori} estimate holds, and that the continuity constant has explicit dependences on $\kappa$ and $\varepsilon$ as long as certain non-trapping conditions are met. We first prove a lemma about controlling the $L^{2}$ norm of $u$, and then use it to determine the   \textit{a-priori} estimate.
\begin{lemma}
\label{lemma1}
If $u \in H_{qp}^{1}(\Omega)$ is a solution to the variational problem \eqref{p1}, then
\begin{equation}
\normLs{u}{\Omega}\leq \bigg[4H\varepsilon_{+}(a+1)+\frac{2H^2}{\min_{\Omega}\big(\frac{1}{\varepsilon} \big)}\bigg]\bigg[\Im \int_{\Gamma_{H}}\overline{u}T^{+}(u)-\Re \int_{\Gamma_{H}} \overline{u} T^{+}(u)+\bigg|\bigg| \bigg(\frac{2}{\epsilon}\bigg)^{1/2}\frac{\partial u}{\partial x_{2}}\bigg| \bigg|_{L^{2}(\Omega)}\bigg],
\end{equation}
where 
\begin{equation}
a=\max_{|\kappa^{2}\varepsilon_{+}-\alpha_{n}^{2}|<1}\bigg(\frac{1}{|\betanplus|} \bigg).
\end{equation}
\end{lemma}
\begin{proof}
By virtue of Lemma 4.3 of Ref.~\cite{Lechleiter}, we know that 
\begin{equation}
\normLs{u}{\Omega} \leq 4H \int_{\Gamma_{H}}\big|u \big|^{2}+4H^{2} \normL{\frac{\partial u}{\partial x_{2}}}{\Omega},
\end{equation}
which is true for all $u \in H^{1}(\Omega)$.
We notice that by Parseval's theorem and by the definition of $\betanplus$,
\begin{align}
\nonumber
\int_{\Gamma_{H}}\big|u \big|^{2} &\leq \sum_{1+\alpha_{n}^{2}<\kappa^{2}\varepsilon_{+}}\big|\betanplus \big| \big|u_{n}(H) \big|^{2}+\sum_{\alpha_{n}^{2}>\kappa^{2}\varepsilon_{+}+1}\big|\betanplus \big| \big|u_{n}(H) \big|^{2}+\sum_{|\kappa^{2}\varepsilon_{+}-\alpha_{n}^{2}|<1}\big|u_{n}(H) \big|^{2}\\
\nonumber
&\leq \varepsilon_{+} \bigg[\Im \int_{\Gamma_{H}}\overline{u} T^{+}(u)-\Re \int_{\Gamma_{H}}\overline{u} T^{+}(u)\bigg]+a \sum_{|\kappa^{2}\varepsilon_{+}-\alpha_{n}^{2}|<1}\big|\betanplus \big|\big|u_{n}(H) \big|^{2}\\
&+\sum_{|\kappa^{2}\varepsilon_{+}-\alpha_{n}^{2}|<1}(1-a\big|\betanplus \big|)\big|u_{n}(H) \big|^{2}.
\end{align}
By virtue of our choice of $a$, $1-a|\betanplus| \leq 0$ 
 and the last sum is non-positive. We add back all the missing terms (where $|\kappa^{2}\varepsilon_{+}-\alpha_{n}^{2}|>1$) into the second to last sum and thus obtain 
\begin{equation}
\normLs{u}{\Omega}\leq 4H\varepsilon_{+}(a+1)\bigg[\Im \int_{\Gamma_{H}}\overline{u} T^{+}(u)-\Re \int_{\Gamma_{H}}\overline{u} T^{+}(u)\bigg]+4H^{2} \normL{\frac{\partial u}{\partial x_{2}}}{\Omega},
\end{equation}
whereby Lemma~\ref{lemma1} is proved.
\qed 
\end{proof}

Our next result is the desired continuity estimate when $\varepsilon$ is smooth.
\begin{theorem}
\label{estimate}
Assume that $\varepsilon \in C^{\infty}(\mathbb{R}^{2})$ is real and $\Re(\epsilon)>0$. Suppose also that the non-trapping conditions 
\begin{enumerate}
\item 
$\frac{\partial}{\partial x_{2}}\big(\frac{1}{\varepsilon} \big)\leq 0 $ in $\Omega$, and 
\item $\varepsilon=\varepsilon_{+}$ on $\Gamma_{H}$
\end{enumerate}
hold. Then, given an $F \in L^{2}(\Omega)$, there exists a unique solution $u \in H_{qp}^{1}(\Omega)$ to the variational problem \eqref{p1} and a continuity constant $C(\kappa,\varepsilon)>0$ such that
\begin{equation}
\normH{u}{\Omega} \leq C(\kappa,\varepsilon) \normL{F}{\Omega}
\end{equation}
with 
\begin{equation}
C(\kappa,\epsilon)=\min \bigg(\min_{\Omega} \bigg(\frac{1}{\varepsilon} \bigg),1 \bigg)^{-1}+(\kappa^{2}+1) \bigg[ 4H\varepsilon_{+}(a+1)+\frac{2H^{2}}{\min_{\Omega}\big(\frac{1}{\varepsilon} \big)}\bigg]\bigg[4H\big(1+\varepsilon_{+}^{1/2}\kappa \big)+2 \bigg].
\label{constant}
\end{equation}
\end{theorem}
\begin{proof}
Using the first non-trapping condition and taking the real part of the Rellich identity, we get
\begin{align}
\nonumber
\int_{\Omega} \frac{2}{\varepsilon} \bigg|\frac{\partial u}{\partial x_{2}} \bigg|^{2}-\Re \int_{\Gamma_{H}} \overline{u} T^{+}(u) &\leq -2 \int_{\Omega}(x_{2}+H) \Re \bigg(\overline{F} \frac{\partial u}{\partial x_{2}} \bigg)-\Re \int_{\Omega} F \overline{u}\\
&+2H\int_{\Gamma_{H}}\bigg(\frac{2}{\varepsilon}\bigg|\frac{\partial u}{\partial x_{2}} \bigg|^{2}-\frac{1}{\varepsilon}\big|\nabla u \big|^{2}+\kappa^{2}\big|u \big|^{2} \bigg). 
\label{E4}
\end{align}
after using the inequality $-\Re \int_{\Gamma_{-H}}\overline{u}T^{-}(u) \geq 0$. Let us recall that $u$ satisfies the same Rayleigh expansion as   provided in \eqref{Rayleigh}. Then,   using Parseval's theorem and the second non-trapping condition, we obtain
\begin{align}
\nonumber
2H\int_{\Gamma_{H}}\bigg(\frac{2}{\varepsilon}\bigg|\frac{\partial u}{\partial x_{2}} \bigg|^{2}-\frac{1}{\varepsilon}\big|\nabla u \big|^{2}+\kappa^{2}\big|u \big|^{2} \bigg)&=2H \sum_{n \in \mathbb{Z}} \bigg( \frac{1}{\varepsilon_{+}} \big|\kappa^{2}\varepsilon_{+}-\alpha_{n}^{2} \big|-\frac{1}{\varepsilon_{+}}\alpha_{n}^{2}+\kappa^{2} \bigg) \big|u_{n}(H) \big|^{2}\\
\nonumber
&=4H \sum_{\alpha_{n}^{2}<\kappa^{2}\varepsilon_{+}}\bigg(\kappa^{2}-\frac{1}{\varepsilon_{+}}\alpha_{n}^{2} \bigg)\big|u_{n}(H) \big|^{2}\\
\nonumber
&=4H\varepsilon_{+}^{-1} \sum_{\alpha_{n}^{2}<\kappa^{2}\varepsilon_{+}} \bigg(\kappa^{2}\varepsilon_{+}-\alpha_{n}^{2} \bigg)\big|u_{n}(H) \big|^{2}\\
&\leq 4H \varepsilon_{+}^{1/2} \kappa \Im \int_{\Omega}F\overline{u}.
\end{align}
This argument is similar to  Lemma~2.2 of Ref.~\cite{Monk}, but we have used different Dirichlet-to-Neumann operators. The inequality $\Im \int_{\Gamma_{H}}\overline{u} T^{+}(u) \leq \Im \int_{\Omega}F \overline{u}$ follows
on setting $u=v$ in the variational problem \eqref{p1} and taking the imaginary part thereof. We combine this result with \eqref{E4} and add $\Im \int_{\Gamma_{H}}\overline{u} T^{+}(u)$ to both sides to obtain
\begin{align}
\nonumber
\int_{\Omega} \frac{2}{\varepsilon} \bigg|\frac{\partial u}{\partial x_{2}} \bigg|^{2}-\Re \int_{\Gamma_{H}} \overline{u} T^{+}(u) +\Im \int_{\Gamma_{H}}\overline{u} T^{+}(u) &\leq -2 \int_{\Omega}(x_{2}+H) \Re \bigg(\overline{F} \frac{\partial u}{\partial x_{2}} \bigg)-\Re \int_{\Omega} F \overline{u}\\  
&+\bigg(4H\varepsilon_{+}^{1/2}\kappa+1\bigg)\Im \int_{\Omega}F\overline{u}.
\label{E5}
\end{align}
Then we combine \eqref{E5} and Lemma~\ref{lemma1} to get
\begin{align}
\nonumber
\normLs{u}{\Omega} &\leq 4H\varepsilon_{+}(a+1)\bigg[\Im \int_{\Gamma_{H}}\overline{u} T^{+}(u)-\Re \int_{\Gamma_{H}}\overline{u} T^{+}(u)\bigg]+4H^{2} \normL{\frac{\partial u}{\partial x_{2}}}{\Omega} \\
\nonumber
&\leq \bigg[4H\varepsilon_{+}(a+1)+\frac{2H^2}{\min_{\Omega}\big(\frac{1}{\varepsilon} \big)}\bigg]\\
\nonumber
&\quad\times \bigg[-2 \int_{\Omega}(x_{2}+H) \Re \bigg(\overline{F} \frac{\partial u}{\partial x_{2}} \bigg)-\Re \int_{\Omega} F \overline{u}
+\bigg(4H(\varepsilon_{+})^{1/2}\kappa+1\bigg)\Im \int_{\Omega}F\overline{u} \bigg]\\
&\leq \bigg[4H\varepsilon_{+}(a+1)+\frac{2H^2}{\min_{\Omega}\big(\frac{1}{\varepsilon} \big)}\bigg]\bigg[4H\big(1+\varepsilon_{+}^{1/2}\kappa \big)+2 \bigg]\normL{F}{\Omega}\normH{u}{\Omega}.
\label{AL-eq35}
\end{align}
After setting $v=u$ in the variational problem \eqref{p1} and taking the real part thereof, we have  
\begin{equation}
\label{AL-eq36}
\normHs{u}{\Omega}\leq \min \bigg( \min_{\Omega}\bigg( \frac{1}{\varepsilon}\bigg),1\bigg)^{-1} \bigg[\normL{F}{\Omega}\normH{u}{\Omega} + (\kappa^{2}+1) \normLs{u}{\Omega} \bigg].
\end{equation}
After first combining \eqref{AL-eq35} and \eqref{AL-eq36} and then dividing the result  by $\normH{u}{\Omega}$,
we obtain  the \textit{a-priori} estimate. Existence and uniqueness of $u$ follow because the \textit{a-priori} estimate implies an inf-sup condition for $B_{\varepsilon}(\cdot,\cdot)$ \cite{Monk}.
\qed
\end{proof}

So far we have not discussed problem \eqref{p2} at all. In the trivial case where $\varepsilon$ is constant, $\varepsilon=\varepsilon_{h}$ and there is nothing new to prove. Generally however, $\varepsilon_{h}$ is only piecewise smooth even if $\varepsilon \in C^{\infty}(\mathbb{R}^{2})$, because it has jumps over the inter-slice boundaries. Therefore, to show existence, uniqueness and an $\textit{a-priori}$ estimate for the $\varepsilon_{h}$ problem (and to cover applications to multilayered devices), we need to allow for coefficients with less smoothness.

\section{\textit{A-priori} estimates for $L^{\infty}$ coefficients}
\label{Linf}
We assumed in Section~\ref{Rellich}  that $\varepsilon \in C^{\infty}(\mathbb{R}^{2})$ and showed that an \textit{a-priori} estimate holds for non-trapping domains. In this section, we  extend the \textit{a-priori} estimates to $\varepsilon\in L^{\infty}(\mathbb{R}^{2})$. Remarkably, the
continuity constant  defined in the forthcoming Lemma can be used even for  general $\varepsilon$, and a general right hand side.  Here we use the technique of Graham et al. \cite{Graham}, but modify their argument slightly for our use. They showed these estimates for an exterior Dirichlet problem, so we check that the results hold for our quasi-periodic problem.  First, we prove an \textit{a-priori} estimate where the right   side lies in the dual space $(H_{qp}^{1}(\Omega))^{'}$, but the $\varepsilon$ is smooth.

\begin{lemma}
\label{dual}
Assume that $\varepsilon \in C^{\infty}(\mathbb{R}^{2})$ is real with $\Re(\varepsilon)>0$ and
that $\varepsilon$ satisfies the non-trapping conditions given in Theorem \ref{estimate}. 
For general data,  $F \in (H_{qp}^{1}(\Omega))^{'}$, let $\tilde{u}\in H_{qp}^{1}(\Omega)$ satisfy 
\begin{equation}
B_{\epsilon}(\tilde{u},v)=F(v)
\end{equation}
for all $v \in H_{qp}^{1}(\Omega)$. Then $\tilde{u}$ exists and is unique; furthermore, 
\begin{equation}
\normH{\tilde{u}}{\Omega}\leq C_{1}(\kappa,\varepsilon) \normblank{F}{\big(H_{qp}^{1}(\Omega)\big)^{'}},
\end{equation}
\end{lemma}
where 
\begin{equation}
C_{1}(\kappa,\varepsilon)=\min \bigg(\min_{\Omega}\bigg(\frac{1}{\varepsilon} \bigg),\kappa^{2} \bigg)^{-1} \bigg[1+2\kappa^{2} C(\kappa,\varepsilon) \bigg].
\end{equation}

\begin{proof}
Define $B_{\varepsilon}^{+}(w,v)=B_{\varepsilon}(w,v)+2\kappa^{2} \int_{\Omega}w \overline{v}$ for all $w \in H_{qp}^{1}(\Omega)$ and $v \in H_{qp}^{1}(\Omega)$. Since the signs of the real and imaginary parts of the Dirichlet-to-Neumann integrals on ${\Gamma_{\pm H}}$   are known, we have 
\begin{equation}
-\Re \int_{\Gamma_{H}}\overline{v}T^{+}(v)-\Re \int_{\Gamma_{-H}}\overline{v}T^{-}(v) \geq 0,
\end{equation}
and the sesquilinear form $B^{+}_{\varepsilon}(\cdot,\cdot)$ is coercive since
\begin{equation}
\big|B^{+}_{\varepsilon}(v,v) \big|\geq \Re \big( B^{+}_{\varepsilon}(v,v) \big) \geq \min \bigg(\min_{\Omega}\bigg(\frac{1}{\varepsilon} \bigg),\kappa^{2} \bigg) \normHs{v}{\Omega}
\end{equation}
for every $v \in H_{qp}^{1}(\Omega).$ 

By virtue of the Lax--Milgram Lemma \cite{Petryshyn,evans}, given an $F \in (H_{qp}^{1}(\Omega))^{'}$ we define $u^{+} \in H_{qp}^{1}(\Omega)$ to be the solution to the problem
\begin{equation}
B^{+}_{\varepsilon}(u^{+},v)=F(v),
\end{equation}
for all $v \in H_{qp}^{1}(\Omega)$; furthermore,   
\begin{equation}
\normH{u^{+}}{\Omega}\leq \min \bigg(\min_{\Omega}\bigg(\frac{1}{\varepsilon} \bigg),\kappa^{2} \bigg)^{-1} \normblank{F}{\big(H_{qp}^{1}(\Omega)\big)^{'}}. 
\end{equation}
First, we let $q \in H_{qp}^{1}(\Omega)$ be the solution to 
\begin{equation}
B_{\varepsilon}(q,v)=2\kappa^{2}\int_{\Omega}u^{+}\overline{v}
\end{equation}
for all $v \in H_{qp}^{1}(\Omega)$, which exists and is unique because $2\kappa^{2}u^{+} \in L^{2}(\Omega)$; then, we apply Theorem \ref{estimate}. To complete the proof, we notice that $B_{\varepsilon}(u^{+}+q,v)=B_{\varepsilon}^{+}(u^{+},v)=F(v)$ for all $v \in H_{qp}^{1}(\Omega).$ Since $\tilde{u}=q+u^{+}$, a solution exists and the \textit{a-priori} estimate shows it is unique. \qed
\end{proof}

Now we extend our \textit{a-priori} results to $L^{\infty}(\mathbb{R}^{2})$ coefficients. For a periodic $\varepsilon \in L^{\infty}(\mathbb{R}^{2})$, we seek a sequence of smooth and periodic $C^{\infty}( \mathbb{R}^{2})$ functions that converge to $\varepsilon$ in the sense of $L^{2}$. To use our previous results, this sequence must be uniformly bounded and each function in the sequence has to satisfy the non-trapping conditions given in Theorem \ref{estimate}. To do this, we first prove the following theorem, which is an analog of Theorem 2.7 of Ref.~\cite{Graham}.

\begin{theorem}
\label{sequence}
Let $\phi \in L^{\infty}(\mathbb{R}^{2})$ be given such that $\phi$ is periodic with period $\Lx$ in $x_{1}$ almost everywhere and there are two constants $\phi_{\text{min}}$ and $  \phi_{\text{max}}\geq \phi_{\text{min}}$
such that
\begin{equation}
\phi_{\text{min}}\leq \phi \leq \phi_{\text{max}},
\end{equation}
almost everywhere in $\Omega$.
We can write 
\begin{equation}
\phi(\bm{x})=\phi_{\text{min}}+\Pi(\bm{x}),
\end{equation}
where $\Pi \in L^{\infty}(\mathbb{R}^{2})$ is almost everywhere $\Lx$-periodic in $x_{1}$. Provided that for all $\tau \geq 0$, $\phi$ is monotonically increasing in the $x_{2}$-direction, i.e.,
\begin{equation}
\label{essinf}
\essinf_{\bm{x} \in \Omega}\bigg[\Pi(\bm{x}+\tau\bm{e}_{2})-\Pi(\bm{x}) \bigg]\geq 0,
\end{equation}
there is a sequence $\phi_{\delta} \in C^{\infty}(\mathbb{R}^{2})$ of 
$\Lx$-periodic functions in $x_{1}$ such that
\begin{enumerate}
\item $\normL{\phi-\phi_{\delta}}{\Omega} \to 0$ as $\delta \to 0$, 
\item $\phi_{\text{min}} \leq \phi_{\delta} \leq \phi_{\text{max}}$, and
\item $\frac{\partial}{\partial x_{2}} \phi_{\delta}  \geq 0$.
\end{enumerate}
\end{theorem}
\begin{proof}
Consider the extended domain $\mathcal{U}=\{\bm{x} \in \mathbb{R}^{2} , \ -\Lx<x_{1}<2\Lx \}$ and define $\psi \in C^{\infty}_{0}(\mathbb{R}^{2})$ as
\begin{equation}
\psi(\bm{x})= \begin{cases} 
      C \exp \bigg( \frac{1}{|\bm{x}|-1} \bigg) & \text{if} \ |\bm{x}|<1, \\
      0 & \text{if} \ |\bm{x}|>1, \\
   \end{cases}
\end{equation}
where we choose $C$ so that $\int_{\mathbb{R}^{2}} \psi =1$. Let $\psi_{\delta}(\bm{x})=\delta^{-2}\psi(\bm{x}/\delta)$ for $\delta>0$. Furthermore, let $\phi_{\delta} \in C^{\infty}(\mathcal{U}_{\delta})$ be defined as
\begin{equation}
\phi_{\delta}(\bm{x})=\phi_{\text{min}}+(\Pi \ast \psi_{\delta})(\bm{x})=\phi_{\text{min}}+\int_{\mathbb{R}^{2}}\Pi(\bm{x}-\bm{y})\psi_{\delta}(\bm{y})\  d \bm{y},
\end{equation}
where  
\begin{equation}
\mathcal{U}_{\delta}=\{\bm{x} \in \mathcal{U}, \ \text{dist}(\bm{x},\partial \mathcal{U})>\delta \}.
\end{equation}  
Using standard properties of mollifiers (e.g., Theorem 7 in Ref.~\cite[Sec.~C.5]{evans}), we have that $\normL{\phi-\phi_{\delta}}{\mathcal{V}} \to 0$ as $\delta \to 0$, for any compact subset $\mathcal{V}$ of $\mathcal{U}$. If we choose $\delta<\Lx/2$ then $\Omega \subset \mathcal{U}_{\delta} \subset \subset \mathcal{U}$, so that $\normL{\phi-\phi_{\delta}}{\Omega}\to 0$ as $\delta \to 0$. The condition $\phi_{\text{min}} \leq \phi_{\delta} \leq \phi_{\text{max}}$ follows from the definition of $\phi_{\delta}$. To finish the proof, we notice that 
\begin{align}
\nonumber
\phi_{\delta}(x_{1}+\Lx,x_{2})&=\phi_{\text{min}}+\int_{|\bm{y}|<\delta}\Pi(\bm{x}-\bm{y}+\Lx\bm{e}_{1})\psi_{\delta}(\bm{y}) \ d\bm{y}\\[5pt]
&=\phi_{\delta}(x_{1},x_{2}),
\end{align}
since $\Pi(\bm{x})$ is an $\Lx$-periodic function in $x_1$. To show that each $\phi_{\delta}$ satisfies the non-trapping condition, we see that
\begin{equation}
(\Pi \ast \psi_{\delta})(\bm{x}+\tau\bm{e}_{2})-(\Pi \ast \psi_{\delta})(\bm{x})\geq \essinf_{\bm{x} \in \Omega}\bigg[\Pi(\bm{x}+\tau\bm{e}_{2})-\Pi(\bm{x}) \bigg] \int_{|\bm{y}|<\delta}\psi_{\delta}(\bm{y}) \ d \bm{y},
\end{equation}
for every $\tau \geq 0$. Since the $\psi_{\delta}$ are positive functions of compact support, this implies that $\frac{\partial}{\partial x_{2}}\phi_{\delta} \geq 0$.  
\qed
\end{proof}

The next result is the main result of this section, and it proves an \textit{a-priori} estimate for our problem with a general source term  and a general non-trapping condition. 
\begin{theorem}
\label{duals}
Given $\varepsilon \in L^{\infty}(\mathbb{R}^{2})$, assume that the generalized non-trapping condition 
\begin{equation}
\essinf_{\bm{x} \in \Omega}\bigg[\tilde{\Pi}(\bm{x}+\tau\bm{e}_{2})-\tilde{\Pi}(\bm{x}) \bigg]\geq 0
\end{equation}
holds for all $\tau\geq 0$, where $\varepsilon(\bm{x})=\varepsilon_{\text{min}}+\tilde{\Pi}(\bm{x})$. Then for $F \in \big(H_{qp}^{1}(\Omega)\big)^{'}$, the solution $u \in H_{qp}^{1}(\Omega)$ of
\begin{equation}
\label{Lax}
B_{\epsilon}(u,v)=F(v)
\end{equation}
for all $v \in H_{qp}^{1}$ exists and is unique; furthermore, 
$$\normH{u}{\Omega}\leq C_{1}(\kappa, \varepsilon)\normblank{F}{\big(H_{qp}^{1}(\Omega)\big)^{'}}.$$
\end{theorem}
\begin{remark}
The generalized non-trapping condition means that $\varepsilon$ is monotonically increasing in the $x_{2}$-direction. We can also prove the same result for $\varepsilon$ monotonically decreasing in the $x_{2}$-direction. 
\end{remark}
\begin{proof}
Since $H_{qp}^{1}(\Omega)=\overline{C_{qp}^{\infty}(\Omega)}$ 
 where the closure is taken in the sense of $H^{1}(\Omega)$, given a $\xi >0$ we can choose a $u_{\xi} \in C_{qp}^{\infty}(\Omega)$ such that 
\begin{equation}
\normH{u-u_{\xi}}{\Omega}< \xi.
\end{equation}
We see that $\phi=\varepsilon$ satisfies the conditions of Theorem \ref{sequence}, and so we have a sequence of smooth and periodic functions $\phi_{\delta} \in C^{\infty} (\mathbb{R}^{2})$ such that $\normL{\phi_{\delta}-\varepsilon}{\Omega} \to 0$ as $\delta \to 0$. The $\phi_{\delta}$ also satisfy the non-trapping conditions of Theorem \ref{estimate}.  For each $\delta>0$, we consider the sesquilinear form $B_{\delta}(w,v)$ defined as in \eqref{p1} but with $\phi_{\delta}$ instead of $\varepsilon$. Then
\begin{equation}
B_{\delta}(w,v)=B_{\varepsilon}(w,v)-\int_{\Omega}\bigg( \frac{1}{\varepsilon}-\frac{1}{\phi_{\delta}}\bigg)\nabla w \cdot \nabla\overline{ v} 
\end{equation}
for all $w \in H_{qp}^{1}(\Omega)$ and $v \in H_{qp}^{1}(\Omega)$.
We also see that 
\begin{equation}
\label{form}
B_{\varepsilon}(u_{\xi},v)=F(v)-B_{\varepsilon}(u-u_{\xi},v)
\end{equation}
for all $v \in H_{qp}^{1}(\Omega)$. Combining the last two equalities with $w=u_{\xi}$, we have 
\begin{equation}
B_{\delta}(u_{\xi},v)=F(v)-B_{\varepsilon}(u-u_{\xi},v)-\int_{\Omega}\bigg(\frac{1}{\varepsilon}-\frac{1}{\phi_{\delta}} \bigg)\nabla u_{\xi}\cdot \nabla\overline{ v}
\end{equation}
for all $v \in H_{qp}^{1}(\Omega)$. 

Let $u$ and $u_{\xi}$ be given,   $u'\in H_{qp}^{1}(\Omega)$   be the solution of the variational problem
\begin{equation}
\label{Lax2}
B_{\delta}(u',v)=F(v)
\end{equation}
for all $v \in H_{qp}^{1}(\Omega)$, and $u'' \in H_{qp}^{1}(\Omega)$ be the solution of the variational problem
\begin{equation}
B_{\delta}(u'',v)=-B_{\varepsilon}(u-u_{\xi},v)-\int_{\Omega}\bigg(\frac{1}{\varepsilon}-\frac{1}{\phi_{\delta}} \bigg)\nabla u_{\xi}\cdot \nabla\overline{ v}
\end{equation}
for all $v \in H_{qp}^{1}(\Omega)$. 
It follows from Lemma \ref{dual} that the solutions $u'$ and $u''$ exist since the right   sides in the two variational problems are in the dual space $(H_{qp}^{1}(\Omega))^{'}$. We can choose a $\delta>0$ small enough so that 
\begin{align}
\nonumber
\normH{u''}{\Omega} &\leq C_{1}(\kappa,\phi_{\delta}) \sup_{\normH{v}{\Omega}=1}\frac{ \big|B_{\epsilon}(u-u_{\xi},v) +\int_{\Omega}\big(\frac{1}{\varepsilon}-\frac{1}{\phi_{\delta}} \big)\nabla u_{\xi}\cdot \nabla\overline{  v} \big|}{\normH{v}{\Omega}} \\
\nonumber
& \leq C_{1}(\kappa,\phi_{\delta}) \bigg[ \gamma \normH{u-u_{\xi}}{\Omega}+\normL{\bigg(\frac{1}{\varepsilon}-\frac{1}{\phi_{\delta}} \bigg)\nabla u_{\xi}}{\Omega} \bigg]\\
\nonumber
&\leq C_{1}(\kappa,\phi_{\delta}) \bigg[ \gamma \xi +\normLinf{\nabla u_{\xi}}{\Omega}\bigg(\frac{1}{\varepsilon_{\text{min}}} \bigg)^{2}\normL{\varepsilon-\phi_{\delta}}{\Omega} \bigg]\\
&\leq C_{1}(\kappa,\phi_{\delta}) (\gamma+1)\xi,
\end{align}
where $\gamma$ is the continuity constant of $B_{\varepsilon}(\cdot,\cdot)$. Finally we have
\begin{align}
\nonumber
\normH{u}{\Omega} &\leq \normH{u-u_{\xi}}{\Omega}+\normH{u_{\xi}}{\Omega}\\
\nonumber
&\leq \xi +\normH{u'}{\Omega}+\normH{u''}{\Omega}\\
&\leq \xi + C_{1}(\kappa,\phi_{\delta})\bigg(\normblank{F}{\big(H_{qp}^{1}(\Omega)\big)^{'}}+(\gamma+1)\xi \bigg)
\end{align}
 for all $\xi>0$.
To complete the proof, we recall that the $\phi_{\delta}$ are uniformly bounded and
that $C_{1}(\kappa,\phi_{\delta}) \leq C_{1}(\kappa,\varepsilon)$ follows from the definition of $C_{1}(\kappa,\varepsilon)$ for all $\delta>0$.
\qed
\end{proof}
\begin{remark}
\label{remark}
For $\varepsilon$ piecewise $C^{2}$ in $\mathbb{R}^{2}$ satisfying the non-trapping condition of Theorem \ref{duals}, it follows that $\varepsilon_{h}$ is piecewise $C^{2}$ in $\mathbb{R}^{2}$ and also satisfies the non-trapping conditions. Thus, the associated problem \eqref{p2} has a unique  $u^{h}$ as its solution. This follows because the same $\textit{a-priori}$ estimate holds for the problems (\ref{p1})
and (\ref{p2}), and the two continuity constants are $C_1(\kappa,\varepsilon)$ and $C_{1}(\kappa,\varepsilon_{h})$. The non-trapping conditions allow us to write $C_{1}(\kappa,\varepsilon)$ explicitly in terms of $\kappa$ and $\varepsilon$. This is important because $C_{1}(\kappa,\varepsilon_{h}) \leq C_{1}(\kappa,\varepsilon)$ for all $h>0$, which follows from this explicit dependence. Therefore, for  the solution of  the problem (\ref{p2}), we have an $\textit{a-priori}$ estimate where the continuity constant is independent of $h$. 
\end{remark}

\section{\textit{A-priori} bounds on the solution}
\label{norm}
To prove convergence we need two additional \textit{a-priori} bounds on the solution.
\begin{theorem}
\label{aest}
Let $\varepsilon$ be piecewise $C^{2}$ in $\mathbb{R}^{2}$, and satisfy the non-trapping condition of Theorem \ref{duals}. Let $u_{F} \in H_{qp}^{1}(\Omega)$ denote the solution of problem \eqref{p1} with $F \in L_{qp}^{2}(\Omega)$ on the right hand  side. Then there is a constant $C_{2}(\kappa,\varepsilon)>0$ and an index $s_{1}$, such that
\begin{equation}
\normHonesone{u_{F}}{\Omega} \leq C_{2}(\kappa,\varepsilon)\normL{F}{\Omega},
\end{equation}
where $s_{1} \in (0,1/2)$.
\end{theorem}
\begin{proof}
We extend the domain $\Omega$ by $\ell$ periods on the left and right, and then above and below by including the infinite half-spaces where $x_{2}>H$ and $x_{2}<-H$. This extended domain is then defined as
\begin{equation}
\Omega^{E}=\{\bm{x} \in \mathbb{R}^{2},\ -\ell \Lx< x_{1} <(\ell+1)\Lx \}.
\end{equation}
As it is also useful to define a circular restricted domain,   we choose an $R>0$ such that
\begin{equation}
\Omega_{R} = \{\bm{x} \in \mathbb{R}^{2}, \big|\bm{x}-(\Lx/2,0) \big|<R \}
\end{equation}
satisfies the set inclusion $\Omega \subset \Omega_{R} \subset \Omega^{E}$. The right   hand side $F$ is extended to $\Omega_{R}$ by quasi-periodicity in $x_{1}$ and by zero above and below. We can also extend the solution $u_{F}$ to the domain $\Omega^{E}$ by quasi-periodicity to the left and right in $x_{1}$, and using the Rayleigh--Bloch expansion  \eqref{Rayleigh} above and below, to obtain $u_{F}^{E} \in H^{1}_{qp}(\Omega^{E})$. 

Let $\chi$ be a smooth cut-off function such that $\chi=1$ in $\Omega$, $\chi=0$ on $\partial \Omega_{R}$ and $|\nabla \chi|<1$ in $\Omega_{R}$. We consider $w=\chi u_{F}^{E}$, and notice immediately that $w=u$ in $\Omega$ and $w=0$ on $\partial \Omega_{R}$. Then $w \in H^{1}(\Omega_{R})$ solves the elliptic problem
\begin{equation}
\label{ellipticproblem}
\left.\begin{array}{ll}
\nabla \cdot \bigg(\frac{1}{\varepsilon}\nabla w \bigg) = F^{*} &\hspace{7 mm} \text{in} \ \Omega_{R}
\\[5pt]
w = 0 &\hspace{7 mm} \text{on} \ \partial \Omega_{R}
\end{array}\right\},
\end{equation}
where the source function $F^{*} \in H^{s-1}(\Omega_{R})$ for all $s \in [0,1/2)$.
To show that this is true, we let $\xi \in H^{1}(\Omega_{R})$ and consider 
\begin{align}
\nonumber
\int_{\Omega_{R}}\nabla \cdot \bigg(\frac{1}{\varepsilon}\nabla w \bigg) \overline{\xi}&=-\int_{\Omega_{R}}\frac{1}{\varepsilon}u_{F}^{E} \nabla \chi \cdot \nabla\overline{ \xi}-\int_{\Omega_{R}}\frac{1}{\varepsilon}\nabla u_{F}^{E} \cdot \nabla(\chi \overline{\xi})\\
\label{eqlist}
&+\int_{\Omega_{R}}\frac{1}{\varepsilon}\nabla u_{F}^{E} \cdot \nabla (\chi) \overline{\xi},
\end{align}
which can be obtained by the divergence theorem and adding and subtracting terms. The first term on the right   side of \eqref{eqlist} can be rewritten for all $\xi \in H^{1}(\Omega_{R})$ as follows:
\begin{equation}
-\int_{\Omega_{R}}\frac{1}{\varepsilon}u_{F}^{E} \nabla \chi \cdot \nabla\overline{ \xi}=\int_{\Omega_{R}}\nabla \cdot \bigg(\frac{1}{\varepsilon}u_{F}^{E} \nabla \chi \bigg) \overline{\xi}.
\end{equation}
Since $\nabla \cdot \bigg(\frac{1}{\varepsilon}\nabla u_{F}^{E} \bigg)=F^{E}-\kappa^{2}u_{F}^{E}$ in $\Omega_{R}$, we find using the divergence theorem that the second term on the right side
of (\ref{eqlist}) may be rewritten as
\begin{equation}
-\int_{\Omega_{R}}\frac{1}{\varepsilon}\nabla u_{F}^{E} \cdot \nabla (\chi \overline{\xi}) = \int_{\Omega_{R}}F^{E}\overline{\xi}\chi - \int_{\Omega_{R}}\kappa^{2} u_{F}^{E} \overline{\xi} \chi,
\end{equation}
for all $\xi \in H^{1}(\Omega_{R})$. The boundary integrals on $\partial \Omega_{R}$ cancel because $\chi \overline{\xi}=0$ on $\partial\Omega_{R}$. Hence, \eqref{eqlist} simplifies to
\begin{equation}
\label{rhs}
\int_{\Omega_{R}}\nabla \cdot \bigg(\frac{1}{\varepsilon}\nabla w \bigg)\overline{\xi}=\int_{\Omega_{R}} \bigg [\nabla \cdot \bigg(\frac{1}{\varepsilon}u_{F}^{E}\nabla \chi \bigg) +F^{E}\chi -\kappa^{2}u_{F}^{E}\chi +\frac{1}{\varepsilon}\nabla u_{F}^{E} \cdot \nabla \chi \bigg]\overline{\xi}
\end{equation}
for all $\xi \in H^{1}(\Omega_{R})$. Then $F^{*}$ is given by the terms enclosed in the square bracket on the right  side of \eqref{rhs}, and we can write $F^{*}=\nabla \cdot \big(\frac{1}{\varepsilon}u_{F}^{E}\nabla \chi \big)+\hat{F}$ where $\hat{F} \in L^{2}(\Omega_{R})$. Since $\varepsilon^{-1}$ is piecewise $C^{2}$, 
$\varepsilon^{-1}$ satisfies the conditions of Proposition 2.1 of Ref.~\cite{Bonito}, i.e., $\varepsilon^{-1} \in L^{\infty}(\Omega_{R})$ and $\nabla \varepsilon^{-1}$ is piecewise $L^{\infty}(\Omega_{R})$. Since the product $u_{F}^{E}\nabla \chi \in H^{s}(\Omega_{R})$,   $\nabla \cdot \big(\frac{1}{\varepsilon}u_{F}^{E}\nabla \chi \big) \in H^{s-1}(\Omega_{R})$ for every $s \in [0,1/2)$. This shows that $w$ solves the elliptic problem given in \eqref{ellipticproblem}, and so we apply Proposition 2.2 of Ref.~\cite{Bernardi} to this problem. It follows that there is a constant $c>0$ and an index $s_{1} \in (0,1/2)$ such that
\begin{align*}
\normHonesone{w}{\Omega_{R}}&\leq c \normblank{F^{*}}{H^{s_{1}-1}(\Omega_{R})} \\
&\leq c \big(\normL{ \frac{1}{\varepsilon}u_{F}^{E}\nabla \chi }{\Omega_{R}} +\normL{\hat{F}}{\Omega_{R}}\big)\\
&\leq c(2 \ell +1) \bigg[1+C_{1}(\kappa,\varepsilon)(1+C(\kappa))(\kappa^{2}+2\normLinf{\varepsilon^{-1}}{\Omega})\bigg]  \normL{F}{\Omega}.
\end{align*}
This follows by repeated use the \textit{a-priori} estimate for $u_{F}$ and 
Theorem~3 of Ref.~\cite{CivilettiMonk}.  To complete the proof, we note that $\normHonesone{u_{F}}{\Omega} \leq \normHonesone{w}{\Omega_{R}}$. \qed
\end{proof}

Using the previous theorem we have the following regularity result for $u$.
\begin{corollary}
\label{ah}
Suppose $\varepsilon$ is piecewise $C^{2}$ in $\mathbb{R}^{2}$ and satisfies the non-trapping condition of Theorem \ref{duals}. Let $u \in H_{qp}^{1}(\Omega)$ be the solution to problem \eqref{p1} where the right hand side is $f=\nabla \cdot \big[(\varepsilon_{+}^{-1}-\varepsilon^{-1})\nabla u^{\text{inc}} \big]$. Then 
\begin{equation}
\normblank{u}{H^{1+s_{1}}(\Omega)}\leq C_{2}(\kappa,\varepsilon) \normL{(\varepsilon^{-1}_{\pm}\Delta+\kappa^{2})(\chi u^{inc})}{\Omega_{\delta}}+\normblank{u^{inc}}{H^{1+s_{1}}(\Omega_{\delta})},
\end{equation}
for some $s_{1}\in(0,1/2)$,  an  appropriately chosen domain $\Omega_{\delta}$, and 
a smooth cut-off function $\chi$.
\end{corollary}

\begin{remark}
As $f=\nabla \cdot \big[(\varepsilon_{+}^{-1}-\varepsilon^{-1})\nabla u^{\text{inc}} \big] \in (H^{1}_{qp}(\Omega))^{'}$ is singular at first sight,   we might only expect $u \in H^{1}_{qp}(\Omega)$. However, the special form of the solution allows for some extra regularity.
\end{remark}
\begin{proof}
By virtue of Theorem \ref{duals}, we know that $u\in H^{1}(\Omega)$ exists and is unique. Given a $\delta>0$, we define an extended domain 
\begin{equation}
\Omega_{\delta}=\{\bm{x}\in \mathbb{R}^{2} , \ 0<x_{1}<\Lx, -H-\delta<x_{2}<H+\delta \}.
\end{equation} 
The top and bottom boundaries of $\Omega_{\delta}$ are $\Gamma_{H+\delta}$ and $\Gamma_{-H-\delta}$, respectively. The smooth cut-off function $\chi$ is defined so that $\chi=1$ in $\Omega$ and $\chi=0$ on $\Gamma_{H+\delta}$ and $\Gamma_{-H-\delta}$. We recall that the total field $u^{t}=u+u^{\rm inc}$ 
and define $\tilde{w}=\chi u^{t}+(1-\chi)u$. Now by definition, $\tilde{w}=u^{t}$ in $\Omega$, and 
\begin{equation}
\label{AL-eq75}
\nabla \cdot \bigg(\frac{1}{\varepsilon}\nabla \tilde{w} \bigg)+\kappa^{2}\tilde{w}=\big(\varepsilon^{-1}_{\pm}\Delta+\kappa^{2}\big)(\chi u^{\rm inc})
\end{equation}
in $\Omega_{\delta}$. As the right   side of (\ref{AL-eq75}) is in $L^{2}(\Omega_{\delta})$,  we apply Theorem \ref{aest} and find a constant $C_{2}(\kappa,\varepsilon)>0$ such that \begin{equation}
\normHonesone{u^{t}}{\Omega} \leq \normHonesone{\tilde{w}}{\Omega_{\delta}} \leq C_{2}(\kappa,\varepsilon)\normL{\big(\varepsilon^{-1}_{\pm}\Delta+\kappa^{2}\big)(\chi u^{\rm inc})}{\Omega_{\delta}}.
\end{equation}
The proof follows by the triangle inequality. \qed
\end{proof}

\section{Convergence of   RCWA   in $h$}
\label{thickness}
We can now prove convergence of the solution $u^h$ of the perturbed problem (see (\ref{p2})):
\begin{theorem}
\label{slice}
Let $\varepsilon$ be piecewise $C^{2}$ in $\mathbb{R}^{2}$ be real with $\Re(\varepsilon)>0$ and satisfy the non-trapping condition of Theorem \ref{duals}. Suppose that the interfaces are the graphs of piecewise $C^{2}$ functions.  Let $u \in H_{qp}^{1}(\Omega)$ be the solution of problem \eqref{p1} with $f=\nabla \cdot \big[(\varepsilon_{+}^{-1}-\varepsilon^{-1})\nabla u^{\text{inc}} \big]$ on the right   side. Also let $u^{h} \in H_{qp}^{1}(\Omega)$ be the solution of problem \eqref{p2} with $f=\nabla \cdot \big[(\varepsilon_{+}^{-1}-\varepsilon_{h}^{-1})\nabla u^{\text{inc}} \big]$ on the right   side. Then there is a constant $C>0$ independent of $h>0$ such that 
\begin{equation}
\normH{u-u^{h}}{\Omega}\leq C h^{s_{1}/2},
\end{equation}
where $ s_{1}\in(0,1)$ is related to the regularity of $u$, i.e., $u \in H^{1+s_{1}}(\Omega)$.
\end{theorem}
\begin{proof}
Since $\varepsilon$ is periodic and piecewise $C^{2}$ in $\mathbb{R}^{2}$, both $\varepsilon$ and $\varepsilon_{h}$ are periodic and are in $L^{\infty}(\mathbb{R}^{2})$. By the definition of $\varepsilon_{h}$, we can write $\varepsilon_{h}(\bm{x})=(\varepsilon_{h})_{\text{min}}+\Pi(\bm{x})$ where $\Pi(\bm{x})=\varepsilon(x_{1},h_{j-\frac{1}{2}})-(\varepsilon_{h})_{\text{min}}$ in slice $S_{j}$. We also write $\varepsilon(\bm{x})=\varepsilon_{\text{min}}+\tilde{\Pi}(\bm{x})$. Then given a $\tau\geq 0$ there is an integer $n \geq 0$ such that
\begin{align}
\Pi(\bm{x}+\tau\bm{e}_{2})-\Pi(\bm{x})&=\varepsilon(x_{1},h_{j-\frac{1}{2}}+nh)-\varepsilon(x_{1},h_{j-\frac{1}{2}}) \nonumber\\
&= \tilde{\Pi}(\bm{x}_{h}+nh\bm{e}_{2})-\tilde{\Pi}(\bm{x}_{h}) \nonumber \\
&\geq \inf_{\bm{x} \in \Omega} \bigg[\tilde{\Pi}(\bm{x}+nh\bm{e}_{2})-\tilde{\Pi}(\bm{x}) \bigg].
\end{align}
Since $\varepsilon$ is monotonically increasing in the $x_{2}$-direction, it follows that 
\begin{equation}
\inf_{\bm{x} \in \Omega} \bigg[\Pi(\bm{x}+\tau\bm{e}_{2})-\Pi(\bm{x}) \bigg]\geq 0.
\end{equation}
Thus, $\varepsilon_{h}$ satisfies the non-trapping conditions of Theorem \ref{duals}.  
We notice that
\begin{equation}
\label{AL-eq80}
B_{\varepsilon_{h}}(u^{t}-u^{h,t},v)=\int_{\Omega}\bigg(\frac{1}{\varepsilon_{h}}-\frac{1}{\varepsilon} \bigg)\nabla u^{t} \cdot \nabla\overline{  v}
\end{equation}
for all $v \in H_{qp}^{1}(\Omega)$, where $u^{h,t}=u^{h}+u^{\text{inc}}$. As the right   side is in the dual space $(H_{qp}^{1}(\Omega))^{'}$,   we apply Theorem \ref{duals} to \eqref{AL-eq80} to see that
\begin{align}
\label{hconv}
\normH{u^{t}-u^{h,t}}{\Omega}& \leq C_{1}(\kappa,\varepsilon_{h}) \sup_{\normH{v}{\Omega}=1} \frac{\big|\int_{\Omega}(\varepsilon_{h}^{-1}-\varepsilon^{-1}) \nabla u^{t} \cdot \nabla \overline{v} \big|}{\normH{v}{\Omega}}\nonumber\\
&\leq C_{1}(\kappa,\varepsilon_{h})\normL{\bigg(\frac{1}{\varepsilon_{h}}-\frac{1}{\varepsilon} \bigg)\nabla u^{t}}{\Omega} \nonumber\\
&\leq C_{1}(\kappa,\varepsilon_{h})\normLinf{(\varepsilon_{h}\varepsilon)^{-1}}{\Omega}\normLp{\varepsilon-\varepsilon_{h}}{\Omega}{2p}\normLp{\nabla u^{t}}{\Omega}{2q},
\end{align}
where $(1/p)+(1/q)=1$, which follows from H{\"{o}}lder's inequality. Using the Sobolev embedding theorem (cf. Theorem~6 of Ref.~\cite[Sec.~5.6.3]{evans}), we choose 
\begin{equation}
\frac{1}{2q}=\frac{1}{2}-\frac{s_{1}}{2},
\end{equation}
which implies that there is a constant $C>0$ independent of $h>0$ such that 
\begin{equation}
\normLp{\nabla u^{t}}{\Omega}{2q} \leq C \normHonesone{u^{t}}{\Omega}. 
\end{equation} 
Then $2p={2}/{s_{1}}$, $2q={2}/{(1-s_{1})}$, and
\begin{equation}
\normH{u^{t}-u^{h,t}}{\Omega}\leq C C_{1}(\kappa,\varepsilon_{h}) \bigg(\frac{1}{\min_{\Omega}\varepsilon} \bigg)^{2} \normLp{\varepsilon-\varepsilon_{h}}{\Omega}{2/s_{1}}\normHonesone{u^{t}}{\Omega}.
\end{equation}
Using  Lemma 6 of Ref.~\cite{CivilettiMonk}, we know that there is a constant $c>0$ independent of $h$ such that 
\begin{equation}
\normLp{\varepsilon-\varepsilon_{h}}{\Omega}{2/s_{1}} \leq c h^{s_{1}/2}.
\end{equation}  
To complete the proof, we recall that $C_{1}(\kappa,\varepsilon_{h})\leq C_{1}(\kappa,\varepsilon)$ for all $h>0$, and $u^{t}-u^{h,t}=u-u^{h}$.
\qed
\end{proof}
%Theorem \ref{slice} thus provides  insights to understand the error arising from approximating $\varepsilon$ by $\varepsilon_{h}$.

\section{RCWA   as a Galerkin scheme}
\label{Galerkin}
Let us now turn our attention to the  error arising from the truncation of all Fourier series. From now on, we denote the RCWA solution by
\begin{equation}
\label{RCWA}
u^{h,M}(x_{1},x_{2})=\sum_{n=-M}^{M} u_{n}^{h,M}(x_{2})\exp(i \alpha_{n}x_{1}),
\end{equation} 
for $\bm{x} \in \Omega$,   the number of retained Fourier modes being $2M+1$. This definition makes sense given that the solution $u$ to the continuous problem  can be written as  \cite{Yakubovich}
\begin{equation}
\label{actual}
u(x_{1},x_{2})=\sum_{n\in\mathbb{Z}} u_{n}(x_{2})\exp(i \alpha_{n}x_{1})
\end{equation}
with coefficient functions $u_{n}(x_{2})$.

Exactly $2M+1$ coefficients have to be ascertained in each   slice $S_{j}$, and the solution 
in the entire domain is reconstructed using \eqref{RCWA}. Therefore, we must determine that the RCWA  converges as $M \to \infty$, and to what order. To do this, we first show that the RCWA is actually a Galerkin scheme; in other words, it solves the variational problem \eqref{p2} where the test functions are in an appropriately chosen finite-dimensional subspace of $H_{qp}^{1}(\Omega)$. We therefore define the space 
\[
V_{M}=H^{1}(-H,H) \otimes E_{M} 
\]
where $E_{M}=\text{span} \{\exp(i \alpha_{n} x_{1}) \ : \ -M\leq n \leq M \}.$

There is some ambiguity in how to formulate the RCWA for $p$-polarized incident light \cite{Tom}. In particular, the Fourier coefficients solve a second-order ODE in each slice, but this ODE can be formulated in several different ways. These formulations and their  numerical convergence properties were discussed by Li~\cite{Li}. In this paper we only consider one of these formulations, i.e., the one that is equivalent to the variational problem.

Let us briefly describe the RCWA, in order to show that it is actually a variational method. In this section we no longer use a general source $F \in (H_{qp}^{1}(\Omega))^{'}$, but instead the incident field is the plane wave  \eqref{incident}.  
Rayleigh--Bloch expansions of the unknown reflected field $u^{\text{ref}}(\bm{x})$, the unknown transmitted field $u^{\text{tr}}(\bm{x})$, and
 the known incident field are used. Thus,  
the incident field is represented in RCWA as
\begin{equation}
u^{\text{inc}}(x_{1},x_{2})=\sum_{n=-M}^{M} u_{n}^{\text{inc}} \exp[-i\betanplus (x_2- H)]
 \exp(i \alpha_{n}x_{1}),\quad
 x_2>H,
\end{equation}
with $u_{n}^{\text{inc}} =0$~$\forall n\ne0$ and $u_{0}^{\text{inc}} =1$;
the reflected field as
\begin{equation}
\label{AL-eq91}
u^{\text{ref}}(x_{1},x_{2})=\sum_{n=-M}^{M} u_{n}^{\text{ref}}
\exp[i\betanplus (x_2-H)]\exp(i \alpha_{n}x_{1}),\quad
 x_2>H;
\end{equation}
and the transmitted field as
\begin{equation}
\label{AL-eq92}
u^{\text{tr}}(x_{1},x_{2})=\sum_{n=-M}^{M} u_{n}^{\text{tr}} \exp[-i\betanminus(x_2+H)]
\exp(i \alpha_{n}x_{1}),\quad
x_2<-H.
\end{equation}
The $2(2M+1)$   coefficients $\left\{u_{n}^{\text{ref}}\right\}_{n=-M}^{n=M}$ and $\left\{u_{n}^{\text{tr}}\right\}_{n=-M}^{n=M}$
 have to be determined.

\begin{theorem}
\label{Galerkin-theorem}
The RCWA solution $u^{h,M} \in V_{M}$ solves the variational problem 
\begin{equation}
\label{p3}
B_{\varepsilon_{h}}(u^{h,M},v_{M})=f(v_{M}),
\end{equation}
for all $v_{M} \in V_{M}$,
where $$f(v_{M})=-\int_{\Omega}\nabla \cdot \big[ (\varepsilon_{+}^{-1}-\varepsilon_{h}^{-1})\nabla u^{\text{inc}} \big] \overline{v_{M}}.$$
\end{theorem}
\begin{proof}
Each Fourier coefficient function $u_{n}^{h,M}(x_{2})$ in \eqref{RCWA} solves   the second-order ODE
(see \cite{Hench})
\begin{equation}
\label{ODE}
\sum_{m=-M}^{M}\varepsilon_{h,n-m}^{-1}\frac{\partial^{2}u^{h,M}_{m}}{\partial x_{2}}=\alpha_{n}\sum_{m=-M}^{M}\varepsilon_{h,n-m}^{-1}\alpha_{m} u_{m}^{h,m}-\kappa^{2}u_{n}^{h,M}+f_{n}
\end{equation}
in each slice $S_{j}$ for $j \in [1,S]$. Here, $\varepsilon_{h,n-m}^{-1}$ is the the $(n,m)$th entry of the Toeplitz matrix formed from the Fourier coefficients of $1/\varepsilon_{h}$. Let $v_M \in V_{M}$, so $v_{M}=\sum_{m=-M}^{M}\xi_{m}(x_{2})\phi_{m}(x_{1})$, where $\xi_{m} \in H^{1}(-H,H)$ and $\phi_{m} \in E_{M}$. We multiply both sides of \eqref{ODE} by $\overline{\xi_n},$ and integrate both sides with respect to $x_{2}$ on $[h_{j-1},h_{j}]$. Integrating the left   side of \eqref{ODE} by parts in $x_{2}$, we have
\begin{align}
\nonumber
\int_{h_{j-1}}^{h_{j}}&\bigg(\sum_{m=-M}^{M}\varepsilon_{h,n-m}^{-1}\frac{\partial u^{h,M}_{m}}{\partial x_{2}}\bigg)
\frac{\partial \overline{\xi_{n}}}{\partial x_{2}} \ d x_{2}-\bigg(\sum_{n=-M}^{M} \varepsilon_{h,n-m}^{-1} \frac{\partial u_{m}^{h,M}}{\partial x_{2}}(h_{j}^{-}) \bigg)\overline{\xi_{n}}(h_{j}^{-})\\ 
\nonumber
&+\bigg(\sum_{n=-M}^{M} \varepsilon_{h,n-m}^{-1} \frac{\partial u_{m}^{h,M}}{\partial x_{2}}(h_{j-1}^{+}) \bigg)\overline{\xi_{n}}(h_{j-1}^{+})\\&=\int_{h_{j-1}}^{h_{j}}\bigg(\kappa^{2}u_{n}^{h,M}-\alpha_{n}\sum_{m=-M}^{M}\varepsilon_{h,n-m}^{-1}\alpha_{m}u_{m}^{h,M} -f_{n} \bigg)\overline{\xi_{n}}\ d x_{2}.
\label{eq1}
\end{align} 
We sum \eqref{eq1} over all $j \in [1,S]$. Since the continuity of $\sum_{m=-M}^{M}\varepsilon_{h,n-m}^{-1}\frac{\partial u_{n}^{h,M}}{\partial x_{2}}$ is enforced across the inter-slice boundaries, every boundary term in \eqref{eq1} cancels except the terms on $\Gamma_{H}$ and $\Gamma_{-H}$. We multiply both sides of the resulting equation by $\phi_{n}\overline{\phi_{n}}$, and integrate both sides with respect to $x_{1}$ in one period. The resulting equation is
\begin{align}
 \nonumber
\int_{\Omega}& \bigg( \sum_{m=-M}^{M} \varepsilon_{h,n-m}^{-1}\frac{\partial u_{m}^{h,M}}{\partial x_{2}} \bigg)\phi_{n} \frac{\partial \overline{\xi_{n}}}{\partial x_{2}} \overline{\phi_{n}}=-\int_{\Gamma_{H}}\bigg(\sum_{m=-M}^{M}\varepsilon_{h,n-m}^{-1} \frac{\partial u_{m}^{h,M}}{\partial x_{2}}\bigg)\phi_{n}\overline{\xi_{n}}\overline{\phi_{n}} \\
\nonumber
&+\int_{\Gamma_{-H}}\bigg(\sum_{m=-M}^{M}\varepsilon_{h,n-m}^{-1} \frac{\partial u_{m}^{h,M}}{\partial x_{2}}\bigg)\phi_{n}\overline{\xi_{n}}\overline{\phi_{n}} \\
\vspace{6pt}
&=\int_{\Omega}\bigg(\kappa^{2}u_{n}^{h,M}-\alpha_{n}\sum_{m=-M}^{M}\varepsilon_{h,n-m}^{-1} \alpha_{m} u_{m}^{h,M}-f_{n}  \bigg)\phi_{n} \overline{\xi_{n}}\overline{\phi_{n}}.
\label{eq2}
\end{align}
To complete the proof, we sum \eqref{eq2} over all $M \leq n \leq M$ and use that $\int_{0}^{L_{x}}\phi_{n}\overline{\phi_{m}}=0$ for all $n \neq m$ and $\frac{\partial}{\partial x_{1}}\phi_{n} \frac{\partial}{\partial x_{1}}\overline{\phi_{n}}=\alpha_{n}^{2}\phi_{n}\overline{\phi_{n}}$ to obtain 
\begin{align}
\nonumber
\int_{\Omega}&\bigg(\frac{1}{\varepsilon_{h}}\nabla u^{h,M} \cdot \nabla \overline{v_{M}}-\kappa^{2}u^{h,M} \overline{v_{M}} \bigg)-\int_{\Gamma_{H}}\frac{1}{\varepsilon_{+}}\frac{\partial u^{h,M}}{\partial x_{2}}\overline{v_{M}}+\int_{\Gamma_{-H}}\frac{1}{\varepsilon_{-}}\frac{\partial u^{h,M}}{\partial x_{2}}\overline{v_{M}}\\
&=-\int_{\Omega}f\overline{v_{M}}.
\end{align}
Finally, we recall that $\frac{1}{\varepsilon_{+}}\frac{\partial u^{h,M}}{\partial x_{2}}=T^{+}(u^{h,M})$ and $\frac{1}{\varepsilon_{-}}\frac{\partial u^{h,M}}{\partial x_{2}}=-T^{-}(u^{h,M})$.
\qed
\end{proof}

\section{Convergence of the RCWA  in $M$}
\label{modes}
To show convergence with increasing number $2M+1$ of retained Fourier modes, we first consider an associated adjoint problem. To this end, for $F \in L_{qp}^{2}(\Omega)$ we seek a $z_{F}^{h} \in H_{qp}^{1}(\Omega)$ such that 
\begin{equation}
\overline{B_{\varepsilon_{h}}(\xi,z_{F}^{h})}=-\int_{\Omega}F \overline{\xi}
\end{equation}
for all $\xi \in H_{qp}^{1}(\Omega)$. For $\varepsilon$ piecewise in  $C^{2}$, this problem has a unique solution in $H_{qp}^{1}(\Omega)$ because of Theorem \ref{duals}. The Galerkin orthogonality 
\begin{equation}
B_{\varepsilon_{h}}(u^{h}-u^{h,M},v_{M})=0
\end{equation}
for all $v_{M} \in V_{M}$ holds because $u^{h}$ solves problem \eqref{p2} and the RCWA solution
$u^{h,M}$ solves \eqref{p3}. Taking $\xi=u^{h}-u^{h,M}$, we have 
\begin{align}
\label{Nitsche}
\nonumber
\normL{u^{h}-u^{h,M}}{\Omega}&\leq \gamma \normH{u^{h}-u^{h,M}}{\Omega}\sup_{F \in L_{qp}^{2}(\Omega)}\bigg( \frac{1}{\normL{F}{\Omega}}\inf_{v_{M}\in V_{M}}\normH{z_{F}^{h}-v_{M}}{\Omega}\bigg) \\
\vspace{6pt}
&\leq C_{2}(\kappa,\varepsilon)\gamma \normH{u^{h}-u^{h,M}}{\Omega}M^{-s_{2}}.
\end{align}
This follows because $\normH{z_{F}^{h}-\mathcal{F}_{M}z_{F}^{h}}{\Omega} \leq M^{-s_{2}}\normblank{z_{F}^{h}}{H^{1+s_{2}}(\Omega)}$, and by Theorem \ref{aest}. 

\begin{theorem}
\label{Mc}
Suppose $\varepsilon$ is piecewise $C^{2}$ in $\mathbb{R}^{2}$, real with $\Re(\varepsilon)>0$ and satisfies the non-trapping condition of Theorem \ref{duals}. Let $u^{h} \in H_{qp}^{1}(\Omega)$ be the solution to problem \eqref{p2} with $f=\nabla \cdot \big[(\varepsilon_{+}^{-1}-\varepsilon_{h}^{-1})\nabla u^{\text{inc}} \big]$ on the right   hand side, and $u^{h,M}$ be the RCWA solution. Then there is a constant $C>0$ independent of $h$ and $M$ such that
\begin{equation}
\normblank{u^{h}-u^{h,M}}{s} \leq C M^{(s-2)s_{2}},
\end{equation}
where $s\in\{0,1\}$, $s_{2}\in(0,1/2)$ is related to the regularity of $u^{h}$ and $M$ is large enough.
\end{theorem}
\begin{proof}
Using Galerkin orthogonality again, it follows that
\begin{equation}
B_{\varepsilon_{h}}(u^{h}-u^{h,M},u^{h}-u^{h,M})=B_{\varepsilon_{h}}(u^{h}-u^{h,M},u^{h}-\mathcal{F}_{M}u^{h}).
\end{equation}
We also have that the sesquilinear form $B_{\varepsilon_{h}}(\cdot,\cdot)$ satisfies the G{\"a}rding inequality 
\begin{equation}
\big| B_{\varepsilon_{h}}(w,w)\big| \geq c\normHs{w}{\Omega}-(\kappa^{2}+1)\normLs{w}{\Omega}
\end{equation}
for all $w \in H_{qp}^{1}(\Omega)$, where $c=\min_{\Omega} \big(\frac{\Re (\varepsilon)}{|\varepsilon|^{2}},1 \big)$.
\end{proof}
We use an argument of Schatz (see \cite{Schatz}) and standard properties of Fourier series to obtain
\begin{equation}
\label{Schatz}
\normH{u^{h}-u^{h,M}}{\Omega} \leq c^{-1}\bigg( \gamma M^{-s_{2}}+(\kappa^{2}+1)\normL{u^{h}-u^{h,M}}{\Omega} \bigg).
\end{equation}
Now by taking $M \geq \big(2(\kappa^{2}+1)\gamma C_{2}(\kappa,\varepsilon)c^{-1}\big)^{1/s_{2}}$ in \eqref{Nitsche}, and combining the result with \eqref{Schatz}, we have that
\begin{equation}
\normH{u^{h}-u^{h,M}}{\Omega}\leq 2c^{-1}\gamma M^{-s_{2}}.
\end{equation}
This completes the proof. \qed
\section{Convergence of the RCWA method in dissipative media}
\label{absorbing}
So far we have only discussed the case where $\varepsilon>0$ is real in $\Omega$. This case is the more mathematically interesting one, because we had to use a Rellich identity along with density arguments
 to demonstrate convergence of the RCWA.  
 
 Now we consider Case~II of Section~\ref{formulation}:   $\varepsilon$ is complex in $\Omega$
 with $\Im(\varepsilon)>c_{1}>0$ and $\Re(\varepsilon)>c_{2}>0$. Then, a part of electromagnetic energy incident on the domain during a finite interval of time is absorbed inside the domain. This case is of   interest for optical modeling of  solar cells \cite{Ahmad-CIGS,Ahmad-AlGaAs,Tom,CivilettiMonk} and absorbing gratings \cite{Kahmann,Fally,Brown}.

 Again, we let $\varepsilon$ be piecewise  $C^{2}$ in $\mathbb{R}^{2}$. The case of $\varepsilon$ complex is easier than the purely real case because $\Im(\varepsilon)>c_{1}>0$ and $\Re(\varepsilon)>c_{2}>0$ ensures that the sesquilinear form  $B_{\varepsilon}(\cdot,\cdot)$ defined in \eqref{p1} is coercive. Therefore, we can use Strang lemmas \cite{Previato} to prove convergence. To show that $B_{\varepsilon}(\cdot,\cdot)$ is coercive in this case, we first prove a lemma.
\begin{lemma}
\label{bound}
Let $w \in H_{qp}^{1}(\Omega)$. Then there is a constant $C>0$ such that
\begin{equation}
\normLs{w}{\Omega}\leq C \bigg(|w|^{2}_{H^{1}(\Omega)}+\Im \int_{\Gamma_{H}}\overline{w} T^{+}(w) \bigg)
\end{equation}
where $\vert\cdot\vert_{H^1(\Omega)}$ is the $H^1$ seminorm.
\end{lemma}
\begin{proof}
By contradiction, suppose there is a sequence $\{w_{n} \}_{n=1}^{\infty}$ in $H_{qp}^{1}(\Omega)$ such that
\begin{enumerate}
\item
$\normL{w_{n}}{\Omega}=1,$ and
\item
$\normL{w_{n}}{\Omega}> n \bigg(|w_{n}|^{2}_{H^{1}(\Omega)}+\Im \int_{\Gamma_{H}}\overline{w_{n}} T^{+}(w) \bigg)$.
\end{enumerate}
As the imaginary part of the Dirichlet-to-Neumann boundary integral is nonnegative 
according to \eqref{Dtn}$_2$, we see by using the two aforementioned properties of the sequence $\{w_{n} \}_{n=1}^{\infty}$ that
\begin{equation}
\frac{1}{n}>|w_{n}|_{H^{1}(\Omega)}^{2}\quad \forall\,n \geq 1.
\end{equation}
Then the sequence is bounded in the $H^{1}(\Omega)$ norm. Furthermore, there is a subsequence $\{w_{n(j)} \}_{n=1}^{\infty}$ that converges weakly in $H^{1}$ to some $q \in H^{1}(\Omega)$ but strongly to $q$ in $L^{2}(\Omega)$. It follows that
\begin{equation}
\normLs{w_{n(j)}}{\Omega}+\normLs{\nabla w_{n(j)}}{\Omega}\to \normLs{q}{\Omega},
\end{equation}
so then $\normL{q}{\Omega}=1$. Since strong convergence implies weak convergence (in particular in $L^{2}$), it follows that $(\nabla w_{n(j)},\nabla p)_{L^{2}(\Omega)}\to (\nabla q, \nabla p)_{L^{2}(\Omega)}$ for all $p \in H^{1}(\Omega)$.  But $\lim_{j\to \infty} \nabla w_{n(j)}=0$, and so $(\nabla q,\nabla p)_{L^{2}(\Omega)}=0$ for all $p \in H^{1}(\Omega)$, and $\normL{\nabla q}{\Omega}=0.$  Thus, $q$ is constant in $\Omega$ and also $q \in H_{qp}^{1}(\Omega)$ since it is a closed subspace of $H^{1}(\Omega)$. 
Then $q$ is a quasi-periodic constant function and $q=0$, if the incidence angle $\theta \neq 0$. Since this is a contradiction, we assume that $\theta=0$. 
 Then the only non-zero Fourier coefficient for $q$ is the zeroth one. 
 
 By construction,
\begin{equation}
\frac{1}{n} > \Im \int_{\Gamma_{H}}\varepsilon_{+}\overline{w_{n(j)}}T^{+}(w_{n(j)})=\sum_{\alpha_{k}^{2}<\kappa^{2}\varepsilon_{+}}\sqrt{\kappa^{2}\varepsilon_{+}-\alpha_{k}^{2}}\big|w_{n(j)}^{(k)}(H) \big|^{2}
\end{equation}
and then it follows that $\lim_{j \to \infty}w_{n(j)}^{(k)}(H) =0$ for all $k$ such that $\alpha_{k}^{2}<\kappa^{2}\varepsilon_{+}$. In particular, $\alpha_{0}^{2}=0<\kappa^{2}\varepsilon_{+}$. By the trace theorem \cite{BrennerScott}, we know that there is a constant $c>0$ such that
\begin{equation}
\normLs{w_{n(j)}-q}{\Gamma_{H}} \leq c^{2} \big( \normLs{w_{n(j)}-q}{\Omega}+|w_{n(j)}|_{H^{1}(\Omega)}^{2} \big).
\end{equation}
It follows from Bessel's inequality that
\begin{equation}
\sum_{k \in \mathbb{Z}}\big|w_{n(j)}^{(k)}(H)-q_{k}(H) \big|^{2} \leq \normLs{w_{n(j)}-q}{\Gamma_{H}}\overset{j \to \infty}{\to} 0,
\end{equation}
but then $q=0$ on $\Gamma_{H}$, since the coefficient $q_{0}(H)=0$. Thus $q=0$ in $\Omega$, contradicting that $\normL{q}{\Omega}=1$.
\qed
\end{proof}

Now we can prove the coercivity of $B_{\varepsilon}(\cdot,\cdot)$.
\begin{corollary}
If $\Im(\varepsilon)>c_{1}>0$ and $\Re(\varepsilon)>c_{2}>0$ for some positive constants $c_{1}$ and $c_{2}$ in $\Omega$, then the sesquilinear form $B_{\varepsilon}(\cdot,\cdot)$ is coercive in $H^{1}(\Omega)$.
\end{corollary}
\begin{proof}
It suffices to show that $\Im B_{\varepsilon}(\cdot,\cdot)$ is coercive in $L^{2}(\Omega)$ and $\Re  B_{\varepsilon}(\cdot,\cdot)$ is coercive in $H^{1}(\Omega)$ in the G{\"a}rding sense. Recalling that $\Im(1/\varepsilon)=-\Im(\varepsilon)/|\varepsilon|^{2}$ and the signs of the imaginary parts of the Dirichlet-to-Neumann boundary integral are known per \eqref{Dtn}, we have
\begin{align}
\nonumber
\big|\Im B_{\varepsilon}(w,w)\big|&\geq \min_{\Omega}\bigg(\frac{\Im (\varepsilon)}{|\varepsilon|^{2}},1 \bigg)\bigg(|w|_{H^{1}(\Omega)}^{2} +\Im \int_{\Gamma_{H}}\overline{w} T^{+}(w) \bigg)\\
& \geq C\min_{\Omega}\bigg(\frac{\Im (\varepsilon)}{|\varepsilon|^{2}},1 \bigg) \normLs{w}{\Omega}, 
\label{imag}
\end{align}
for all $w \in H_{qp}^{1}(\Omega)$ per Lemma \ref{bound}. To see that $|\Re B_{\varepsilon}(\cdot,\cdot)|$ is coercive in $H^{1}(\Omega)$ in the G{\"a}rding sense, we recall that $\Re(1/\varepsilon)=\Re(\epsilon)/|\varepsilon|^{2}$, and
\begin{align}
\big| \Re B_{\varepsilon}(w,w)\big| \geq \min_{\Omega} \bigg(\frac{\Re (\varepsilon)}{|\varepsilon|^{2}},1 \bigg)\normHs{w}{\Omega}-(\kappa^{2}+1)\normLs{w}{\Omega}
\end{align}
for all $w \in H_{qp}^{1}(\Omega)$. To show this implies coercivity, there are two cases to consider. We let $\lambda(\varepsilon)=C\min_{\Omega}\bigg(\frac{\Im (\varepsilon)}{|\varepsilon|^{2}},1 \bigg),$ $\sigma(\varepsilon)=\min_{\Omega} \bigg(\frac{\Re (\varepsilon)}{|\varepsilon|^{2}},1 \bigg)$ and $\gamma=\kappa^{2}+1$. If $\sigma \normHs{w}{\Omega}-\gamma \normLs{w}{\Omega}\leq 0$, then 
 \begin{align}
 \nonumber
 \lambda \sigma \normHs{w}{\Omega}&\leq \gamma \lambda \normLs{w}{\Omega}\\ &\leq \gamma \big| \Im B_{\varepsilon}(w,w) \big|,
 \end{align}
  follows from \eqref{imag}. If $\sigma \normHs{w}{\Omega}-\gamma \normLs{w}{\Omega}>0$, we let $c=\lambda^{2}/(2\gamma^{2})$ so that $\lambda^{2}-c\gamma^{2}>0$ and set $\tilde{c}=2(c+1)$. By defining $a=\tilde{c}\gamma^{2}\normL{w}{\Omega}^{4}$ and $b=4(\tilde{c})^{-1}\sigma^{2}\normH{w}{\Omega}^{4}$, it follows from the inequality of arithmetic and geometric means that
 \begin{align}
 \nonumber
 \big|B_{\varepsilon}(w,w) \big|^{2}&=\big|\Im B_{\varepsilon}(w,w) \big|^{2}+\big|\Re B_{\varepsilon}(w,w) \big|^{2}\\
 \nonumber
 &\geq \lambda^{2}\normL{w}{\Omega}^{4}+\sigma^{2}\normH{w}{\Omega}^{4}+\gamma^{2}\normL{w}{\Omega}^{4}-2\sigma \gamma \normH{w}{\Omega}^{2}\normL{w}{\Omega}^{2}\\
 \nonumber
 &\geq \big[\lambda^{2}+\gamma^{2}(1-\tilde{c}/2) \big]\normL{w}{\Omega}^{4}+\sigma^{2}\big[1-2(\tilde{c})^{-1} \big]\normH{w}{\Omega}^{4}.\\
 \nonumber
 &=\frac{\lambda^{2}}{2}\normL{w}{\Omega}^{4}+\sigma^{2}\big(1-\frac{1}{c+1}\big)\normH{w}{\Omega}^{4}\\
 &\geq \frac{\sigma^{2} \lambda^{2}}{\lambda^{2}+2\gamma^{2}} \normH{w}{\Omega}^{4}.
 \end{align}
It follows in either case that
 \begin{equation}
 \label{unib}
 \big|B_{\varepsilon}(w,w) | \geq \frac{\sigma(\varepsilon)}{\sqrt{1+2(\frac{\gamma}{\lambda(\varepsilon)})^{2}}} \normH{w}{\Omega}^{2}
 \end{equation}
 for all $w \in H_{qp}^{1}(\Omega)$. 
\qed
\end{proof} 

Existence and uniqueness of the variational problem \eqref{p1} now follow by virtue of the Lax--Milgram Lemma \cite{evans}. Problem \eqref{p2} has the same characteristics but $\varepsilon_{h}$ is constructed by sampling the true relative permittivity at the inter-slice boundaries. Since $\Im(\varepsilon)>c_{1}>0$ and $\Re(\varepsilon)>c_{2}>0$, it follows that $\Im (\varepsilon_{h})>c_{1}>0$ and $\Re (\varepsilon_{h})>c_{2}>0$ for all $h>0$, and we have existence and uniqueness for the $\varepsilon_{h}$ problem as well. To obtain an \textit{a-priori} estimate where the continuity constant is explicit, we note that
 \begin{equation}
\label{Lax2}
\normH{u^{h}}{\Omega}\leq \frac{\sqrt{1+2(\frac{\gamma}{\lambda(\varepsilon)})^{2}}}{\sigma(\varepsilon)} \normblank{f}{\big(H^{1}_{qp}(\Omega)\big)^{'}}.
\end{equation}
follows from the Lax--Milgram Lemma, and the fact that $\lambda(\varepsilon_{h})^{-1} \leq \lambda(\varepsilon)^{-1} $ and $\sigma(\varepsilon_{h})^{-1} \leq \sigma(\varepsilon)^{-1} $ for all $h>0$.

The discussion above leads us to the following theorem.
 \begin{theorem}
 \label{hconv1}
 Suppose that $\varepsilon$ is piecewise $C^{2}$ in $\mathbb{R}^{2}$,
 $\Im (\varepsilon)>c_{1}>0$, $\Re (\varepsilon)>c_{2}>0$ in $\Omega$,  and the interfaces are graphs of piecewise $C^{2}$ functions.
   Then the two variational problems \eqref{p1} and \eqref{p2} have unique solutions $u$ and $u^{h} \in H_{qp}^{1}(\Omega)$, respectively. Furthermore there is a constant $C>0$ independent of $h>0$ such that
 \begin{equation}
 \normH{u-u^{h}}{\Omega}\leq C h^{s_{1}/2},
 \end{equation}
 where $s_{1}\in(0,1/2)$ is related to the regularity of $u$, i.e., $u \in H^{1+s_{1}}(\Omega)$.
 \end{theorem}
 \begin{proof}
 By virtue of \eqref{unib}, the family of sesquilinear forms $\big(B_{\varepsilon_{h}} (\cdot,\cdot)\big)_{h>0}$ is uniformly $H_{qp}^{1}(\Omega)$-elliptic. According to Strang's first lemma \cite{BrennerScott}, there is a constant $c>0$  independent of $h>0$ such that
 \begin{align}
 \nonumber
 \normH{u^{t}-u^{h,t}}{\Omega}&\leq c \inf_{w \in H_{qp}^{1}(\Omega)}\bigg(\normH{u^{t}-w}{\Omega} +\sup_{v \in H_{qp}^{1}(\Omega)} \frac{\big|B_{\varepsilon}(w,v)-B_{\varepsilon_{h}}(w,v) \big|}{\normH{v}{\Omega}}\bigg)\\
 &\leq c \bigg(\sup_{v \in H_{qp}^{1}(\Omega)} \frac{\big|B_{\varepsilon}(u^{t},v)-B_{\varepsilon_{h}}(u^{t},v) \big|}{\normH{v}{\Omega}}\bigg),
 \end{align} 
 where we put $u^{t}=w$. We then bound the consistency error
 \begin{equation}
 \big|B_{\varepsilon}(u^{t},v)-B_{\varepsilon_{h}}(u^{t},v) \big| \leq \bigg(\frac{1}{\min_{\Omega}|\varepsilon|} \bigg)^{2}\normLp{\varepsilon-\varepsilon_{h}}{\Omega}{2p}\normLp{\nabla u^{t}}{\Omega}{2q}\normH{v}{\Omega},
 \end{equation}
 where $(1/q) + (1/p)=1$. The proof follows just like for Theorem \ref{slice}.
 \qed
 \end{proof}
 
 \begin{theorem}
 \label{c2}
Suppose that  $\varepsilon$ is piecewise $C^{2}$ in $\mathbb{R}^{2}$,
 $\Im (\varepsilon)>c_{1}>0$, and $\Re (\varepsilon)>c_{2}>0$ in $\Omega$. 
 Let $u^{h} \in H_{qp}^{1}$ be the solution to problem \eqref{p2} and $u^{h,M}$ be the RCWA solution. Then there exists a constant $C>0$ independent of $h$ and $M$ such that
\begin{equation}
\normblank{u^{h}-u^{h,M}}{s}\leq C M^{(s-2)s_{2}},
\end{equation} 
where $s\in\left\{0,1\right\}$ and $s_2\in(0,1/2)$ 
is chosen so that $u^{h} \in H^{1+s_{2}}(\Omega)$.
 \end{theorem}
 \begin{proof}
By virtue of \eqref{unib}, the family of sesquilinear forms $\big(B_{\varepsilon_{h}} (\cdot,\cdot)\big)_{h>0}$ is uniformly $V_{M}$-elliptic. The sesquilinear form is the same for both problems, and if we consider the total fields, the variational problems have the same right  hand side.  Strang's first lemma \cite{BrennerScott} yields a constant $c>0$ independent of $h$ and $M$ such that
\begin{align}
\nonumber
\normH{u^{h,t}-u^{h,M,t}}{\Omega}&\leq c \bigg(\inf_{v_{M}\in V_{M}} \normH{u^{h,t}-v_{M}}{\Omega}\bigg)\\
\nonumber
&\leq c \bigg(\normH{u^{h,t}-\mathcal{F}_{M}u^{h,t}}{\Omega}\bigg)\\
\nonumber
&\leq c \bigg(M^{-s_{2}}\normblank{u^{h,t}}{H^{1+s_{2}}(\Omega)}\bigg)\\
&\leq c C_{3}(\kappa,\varepsilon)\normL{\big(\varepsilon^{-1}_{\pm}\Delta+\kappa^{2}\big)(\chi u^{\rm inc})}{\Omega_{\delta}} M^{-s_{2}}.
\end{align}
 Here, $C_{3}(\kappa,\varepsilon)$ is the same as $C_{2}(\kappa,\varepsilon)$ but with the continuity constant in \eqref{Lax2} instead of $C_{1}(\kappa,\varepsilon)$.  We have used standard properties of Fourier series \cite{Walker} in the third line, and the last line follows from Corollary \ref{ah}. The extra order of convergence in the $L^{2}$ norm follows from the Aubin--Nitsche trick \cite{Hsiao}.
 \qed
 \end{proof}
 We now summarize the convergence results of this paper:
 \begin{theorem}
 \label{summary}
 Suppose the conditions of Theorem \ref{slice} or Theorem \ref{hconv1} are met, $u \in H_{qp}^{1}(\Omega)$ is the solution to problem \eqref{p1}, and $u^{h,M}$ is the RCWA solution. For $M$ large enough, there is a constant $C>0$ independent of $h$ and $M$ such that
 \begin{equation}
 \normblank{u-u^{h,M}}{s}\leq C \bigg(h^{s_{1}/2}+M^{(s-2)s_{2}} \bigg),
 \end{equation}
 with $s\in\left\{0,1\right\}$, where $s_{1}$ is related to the regularity of $u$ and $s_{2}$ is related to the regularity of $u^{h}$.
 \end{theorem}
 \begin{proof}
 This follows from Theorems \ref{slice} and \ref{Mc} or Theorems \ref{hconv1} and \ref{c2}, and the triangle inequality.
 \qed
 \end{proof}

\section{Numerical examples}
\label{numerics}
In this section, for two different gratings, we compute the convergence rate of the   RCWA solution by comparing the RCWA solution with the solution obtained
using the Finite Element Method with a highly refined grid, since an analytical solution is not possible to obtain. Both gratings  have a period $\Lx=500$~nm and a triangular profile of base $Lx/2$.
 The first grating, shown in Fig.~\ref{illustration}(a),  has a symmetric triangular profile of height $100$~nm and sits atop a 100-nm-thick strip.  The second grating, shown in Fig.~\ref{illustration}(b),  has an asymmetric triangular profile of height $50$~nm with its vertex shifted off-center by $\Lx/8$
 and sits atop a 50-nm-thick strip. The grating and the underlying strip are made
 of either a metal with relative permittivity $\varepsilon_{m}=-15+4i$ or a dissipative material with 
 relative permittivity $\varepsilon_{d}=15+4i$.
 The domain height $2H =1700$~nm in Fig.~\ref{illustration}(a) but $2H =1600$~nm in Fig.~\ref{illustration}(b). Whatever portion of $\Omega$ is not occupied by metal or dissipative material is filled with air with relative permittivity $\varepsilon_{a}=1+10^{-6}i$. The small imaginary part $\Im(\varepsilon_{a})=10^{-6}$ was used for the sake of numerical stability  of the RCWA algorithm  \cite{Faryad}.  Calculations were made for normal incidence (i.e., $\theta=0$) at free-space wavenumber $\lambda_0 = 2\pi/\kappa=600$~nm.

%%%%%%%%%%%%%%%%% Figure 2 begins %%%%%%%%%%%%%%%%%
\begin{figure}%
    \centering
    \subfloat[]{{\includegraphics[width=7cm]{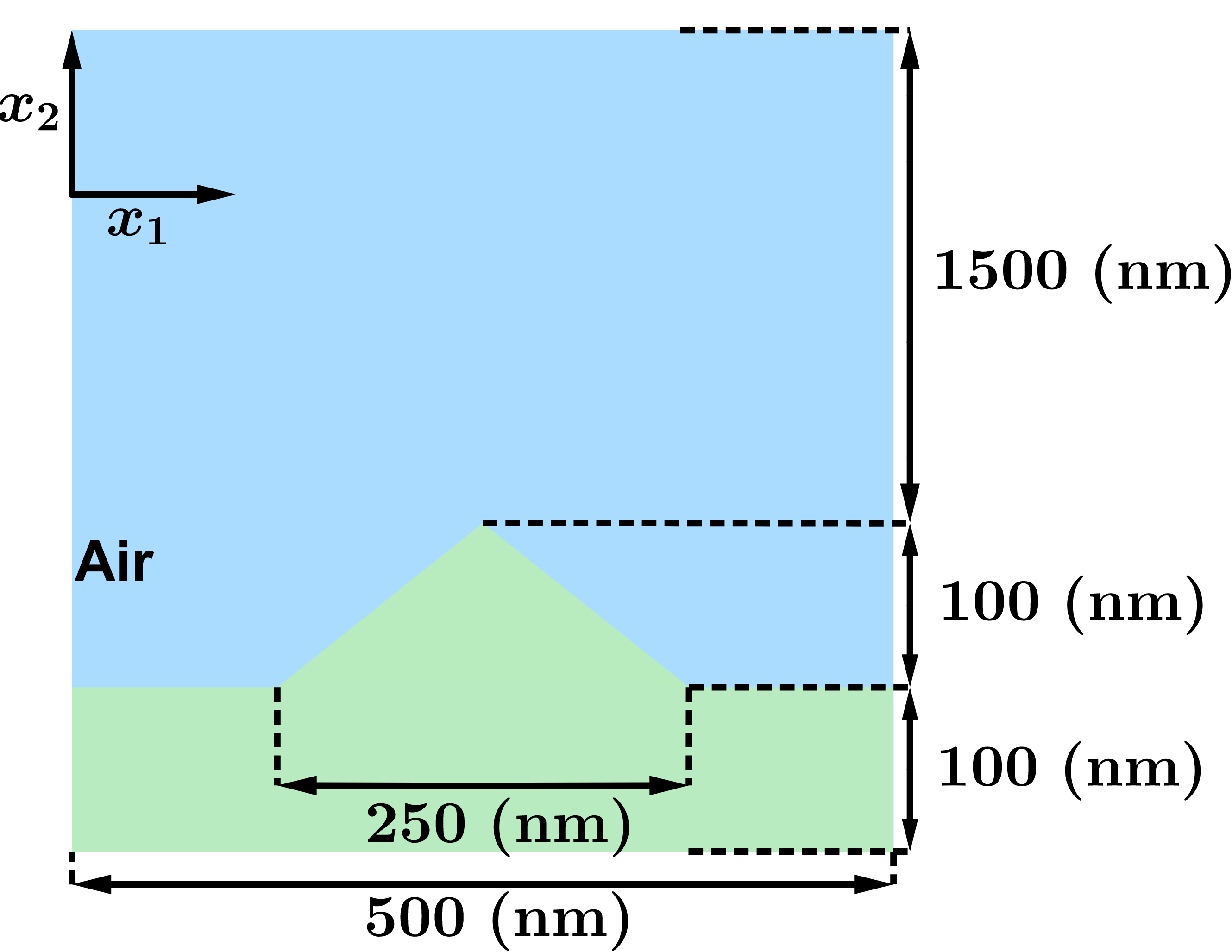} }}%
    \qquad
    \subfloat[]{{\includegraphics[width=7cm]{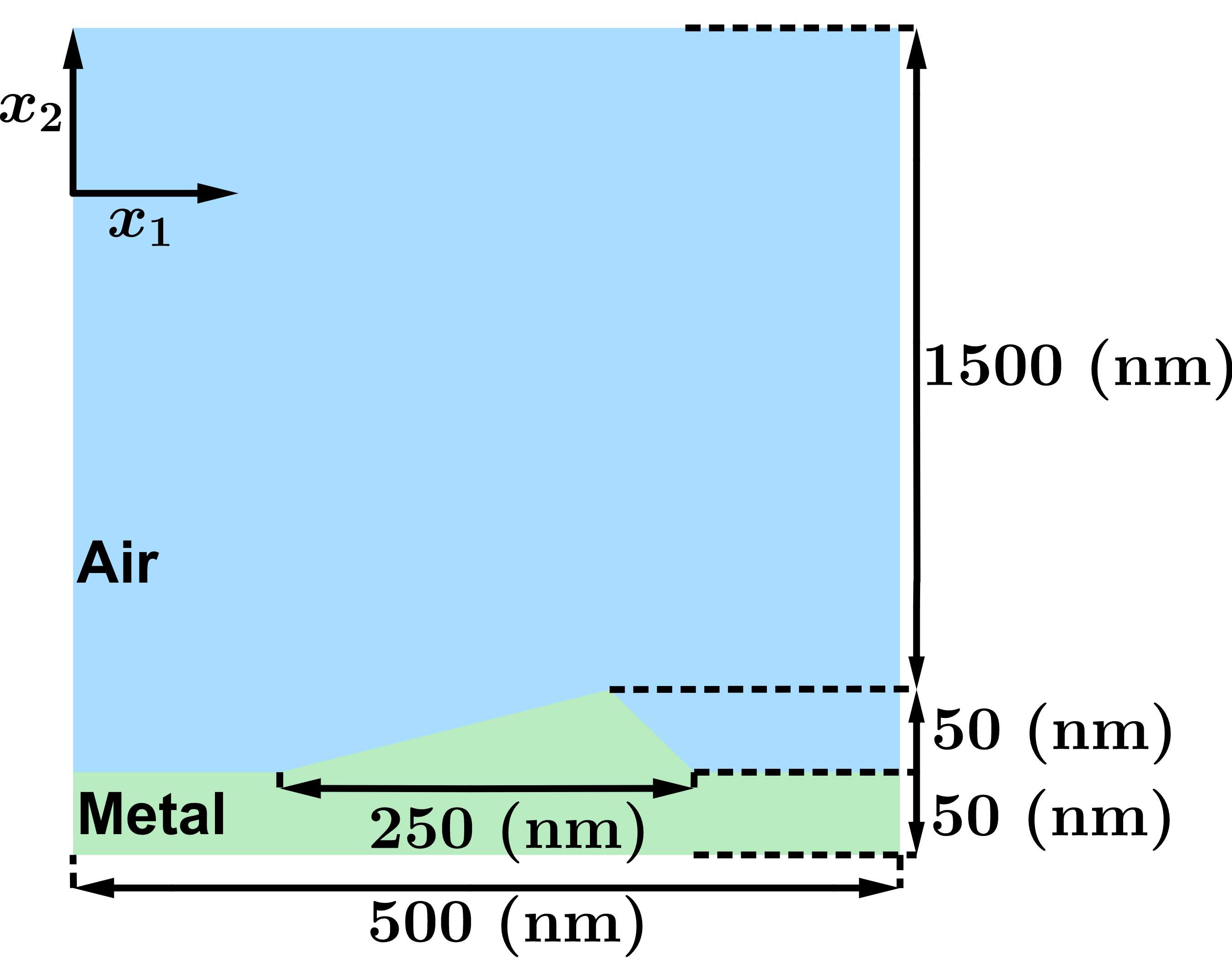} }}%
    \caption{
 (a) Symmetric grating of maximum height  $100$~nm.
 (b) Asymmetric grating of maximum height 50~nm. 
 The peak of the asymmetric grating is off center to the right by $62.5$ nm. The thickness of the air layer is not to scale.
 }%
    \label{illustration}
\end{figure}
%%%%%%%%%%%%%%%%% Figure 2 ends %%%%%%%%%%%%%%%%%

%%%%%%%%%%%%%%%%%%% Figure 3 begins %%%%%%%%%%%%%%%%%
\begin{figure}[h]%
    \centering
    \subfloat[]{{\includegraphics[width=7.8cm]{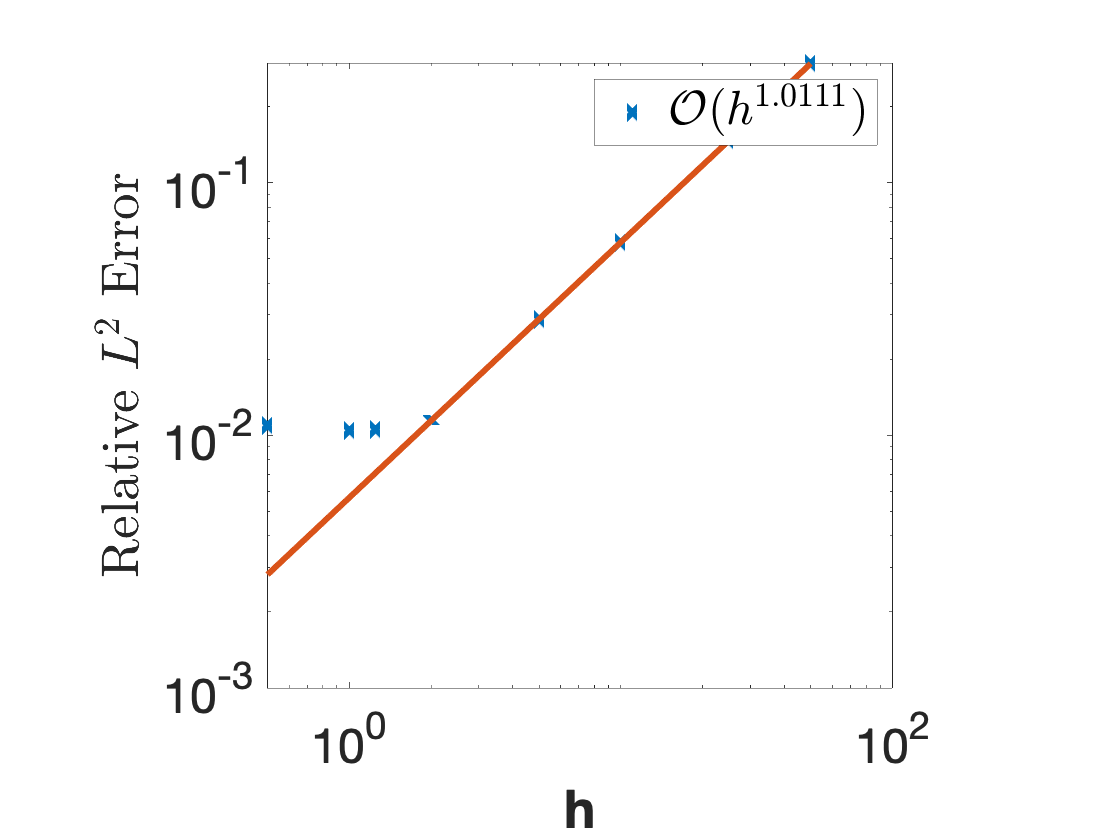} }\label{fig:convergence2:a}}%
    \subfloat[]{{\includegraphics[width=7.8cm]{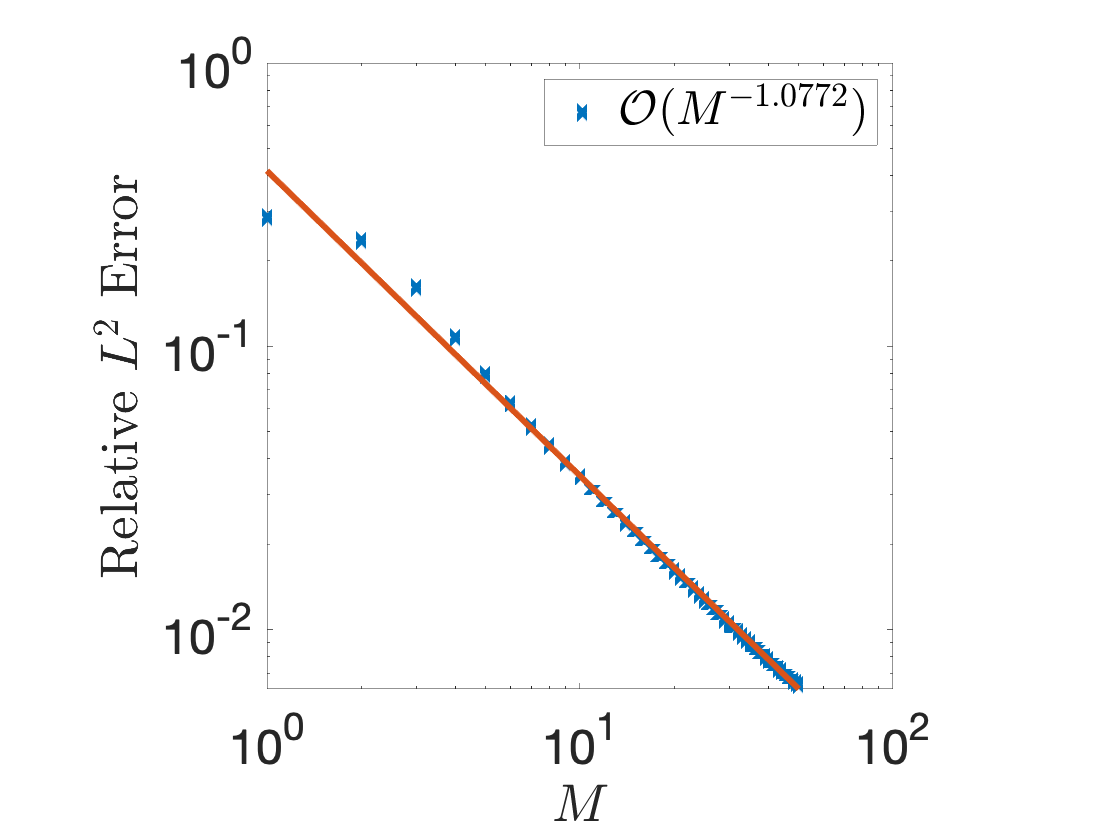} }\label{fig:convergence2:b}}%
    \caption{ Convergence plots comparing the relative $L^{2}$ error between the RCWA and FEM solutions for 
        the  symmetric grating of Fig.~\ref{illustration}(a).
    Whereas $h \in \{1/2,1,1.25,2,5,10,25,50 \}$~nm but   $M=30$ in (a),
    $h=1$~nm but $M=[1,50]$    in (b).  The grating and underlying strip are made of a dissipative material with relative permittivity   $\varepsilon_{d}=15+4i$. In all plots the error saturates below $10^{-2}$. For the least-squares-fit lines,    data in the convergent regime only where used.}%
    \label{fig:convergence2}%
\end{figure}

%%%%%%%%%%%%%%%%%%% Figure 3 ends %%%%%%%%%%%%%%%%%

The FEM solution $u_{FE}$  was computed using an adaptive method in NGSolve version 6.2.1908 \cite{NGSolve}. The domain $\Omega$ was taken to be sandwiched between two perfectly matched layers (PMLs) of thickness equal to $\lambda_0$ and PML parameter equal to $1.5+2.5i$. The FEM solutions were computed using $4$th-degree continuous finite elements, with the mesh adaptivity terminating when the algorithm would reach 100,000 degrees of freedom. The relative $L^{2}$ error between the RCWA  and   FEM solutions was calculated as 
 \begin{equation}
 \frac{\normL{u^{h,M}-u_{FE}}{\Omega}}{\normL{u_{FE}}{\Omega}}.
\end{equation} 

The convergence results in Figs.~\ref{fig:convergence2}(a) and \ref{fig:convergence2}(b) are for the symmetric case of Fig.~\ref{illustration}(a) with $\varepsilon_{d}=15+4i$. Since $\Re(\varepsilon)>0$ and $\Im(\varepsilon)>0$, this case is covered by Theorem \ref{summary}.  Again, as our analysis predicted the convergence rate with respect to $M$ to be $\mathcal{O}(M^{-2s_{2}})$ with $s_2\in(0,1/2)$. In fact, we know that $u^{h} \in H^{1+s_{2}}$ for every $s_{2} \in (0,1/2)$ which follows from Corollary 3.2 of Ref.~\cite{elschner}.  The numerical results yielding $\mathcal{O}(M^{-1.0772})$  match the prediction. 

Theorem \ref{summary} predicts the convergence rate with respect to $h$ to be $\mathcal{O}(h^{s_{1}/2})$. In practice, however, we see faster convergence than predicted by our analysis. This was also observed for s-polarized light \cite{CivilettiMonk}. It would be desirable to improve the predicted convergence rates with respect to $h$ to closely match the higher rates seen in our numerical results. The RCWA converges in a stable way in this case, as the error data closely falls on the least-squares-fit line. 

Figs.~\ref{fig:convergence}(a) and \ref{fig:convergence}(c) show the convergence of the RCWA solution for the symmetric grating with respect to $M$ and $h$, respectively, when the grating and the underlying strip are made of a metal with relative permittivity of $\varepsilon_{m}=-15+4i$. We observe order $\mathcal{O}(M^{-1.028})$ and $\mathcal{O}(h^{0.73068})$. Figs.~\ref{fig:convergence}(b) and \ref{fig:convergence}(d) show the convergences
of the RCWA solution of the asymmetric grating as being order $\mathcal{O}(M^{-0.91695})$ and $\mathcal{O}(h^{0.92497})$, respectively. Although our analysis in the previous sections did not cover the case where $\Re(\varepsilon)<0$,  we included these numerical results because metal gratings are a common application for RCWA.

%%%%%%%%%%%%%%%%%%% Figure 4 begins %%%%%%%%%%%%%%%%%
\begin{figure}[h]%
    \centering
    \subfloat[]{{\includegraphics[width=7.8cm]{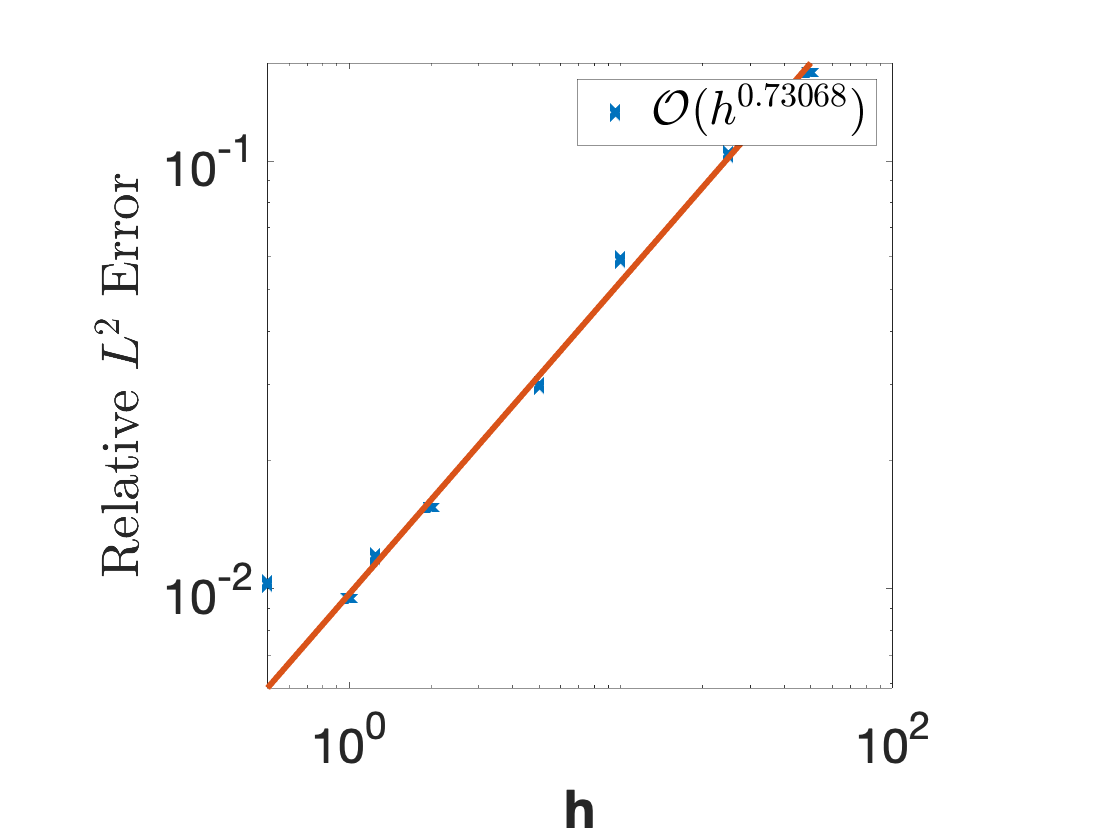} }\label{fig:convergence:a}}%
    \subfloat[]{{\includegraphics[width=7.8cm]{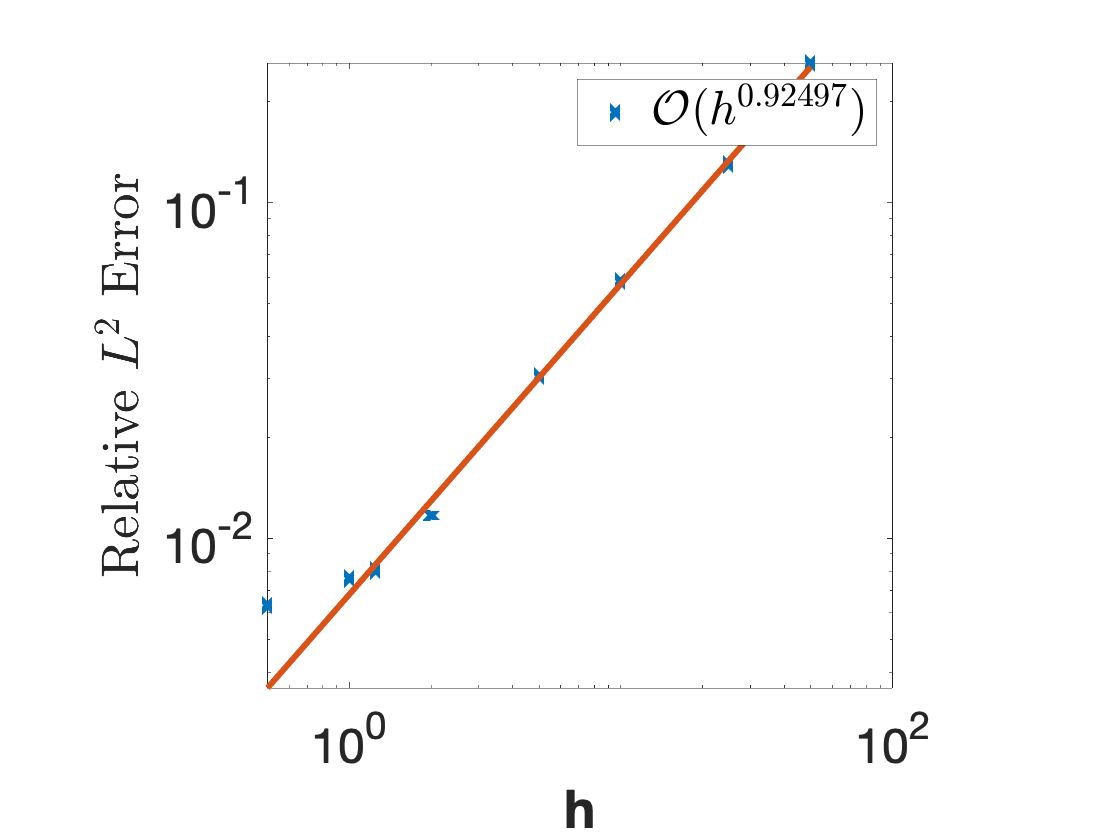} }\label{fig:convergence:b}}%
    \quad
    \subfloat[]{{\includegraphics[width=7.8cm]{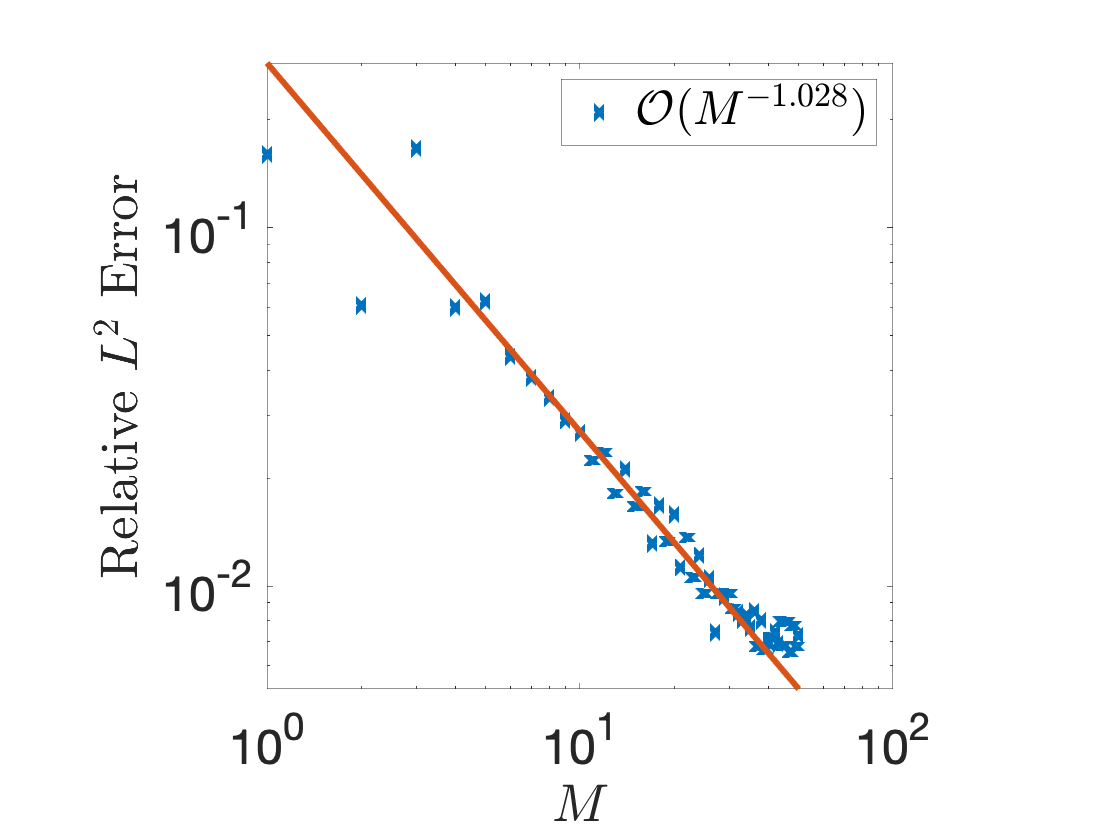} }\label{fig:convergence:c}}%
    \subfloat[]{{\includegraphics[width=7.8cm]{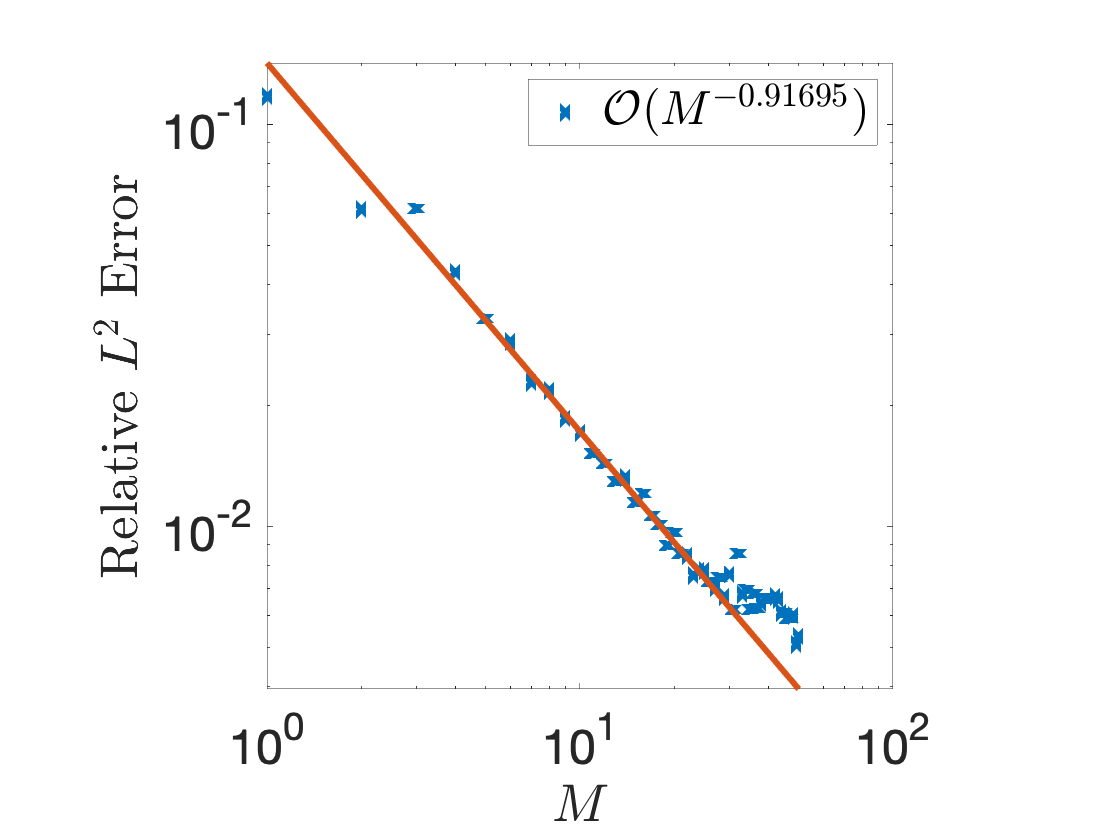} }\label{fig:convergence:d}}%
    \caption{ Convergence plots comparing the relative $L^{2}$ error between the RCWA and FEM solutions for 
       (a,c) the  symmetric grating of Fig.~\ref{illustration}(a) and 
    (b,d) the asymmetric grating of Fig.~\ref{illustration}(b).
    Whereas $h \in \{1/2,1,1.25,2,5,10,25,50 \}$~nm but   $M=30$ in (a) and (b),
    $h=1$~nm but $M=[1,50]$    in (c) and (d).  The grating and underlying strip are made of a metal with relative permittivity   $\varepsilon_{m}=-15+4i$. In all plots the error saturates below $10^{-2}$. For the least-squares-fit lines,    data in the convergent regime only where used.}%
    \label{fig:convergence}%
\end{figure}

%%%%%%%%%%%%%%%%%%% Figure 4 end %%%%%%%%%%%%%%%%%

\section{Conclusion}
The convergence properties of the two-dimensional RCWA   for $p$-polarized   light was studied for domains with piecewise smooth relative permittivity that is either
\begin{itemize}
\item[I.] purely real and positive, or
\item[II.] complex with both real and imaginary parts positive.
\end{itemize}
Since the RCWA approximates   the solution using slices of thickness $h$ and a Fourier truncation parameter $M$, we provided theorems about the convergence rates with
respect to both $h$ and $M$. We showed that the RCWA is a Galerkin scheme for $p$-polarized
light, which allowed investigation using FEM techniques.  If the relative permittivity is purely real, we proved a Rellich identity for non-trapping domains. Our theory predicts convergence even for trapping domains, as long as the continuity constant \eqref{constant} remains bounded independently of $h$. In the case where $\Im(\varepsilon)>c_{1}>0$ and $\Re(\varepsilon)>c_{2}>0$, we proved convergence using  Strang lemmas.

Our study has two limitations which need to be overcome in future work.  We need to extend
our analysis to the case where $\Re(\epsilon)<0$ in some parts of the domain.  This would allow us to assert convergence for metallic scatterers as seen in Fig.~\ref{illustration}.  Secondly we need to allow for absorption in only some parts of the domain when $\Re(\epsilon)>0$.

\newpage
\appendix
\section{Appendix}
\subsection{Proof of Theorem 1}
 
Since $\varepsilon \in C^{\infty}(\mathbb{R}^{2})$,
 $u \in H^{2}(\Omega)$.  Using the identity 
 \begin{equation}
 2 \Re \bigg(\frac{\partial u}{\partial x_{2}} \overline{u} \bigg)=\frac{\partial}{\partial x_{2}} \big|u \big|^{2},
 \end{equation}
 and the Helmholtz equation $\nabla \cdot \bigg(\frac{1}{\varepsilon}\nabla \overline{u} \bigg)=\overline{F}-\kappa^{2}\overline{u}$, we obtain
\begin{align}
\label{parts}
2 \Re \int_{\Omega}(x_{2}+H)\frac{\partial u}{\partial x_{2}}\nabla \cdot \bigg(\frac{1}{\varepsilon}\nabla \overline{u}\bigg)&=\int_{\Omega}(x_{2}+H) 2 \Re \bigg( \frac{\partial u}{\partial x_{2}}\overline{F}\bigg)-\kappa^{2}\int_{\Omega}(x_{2}+H)\frac{\partial}{\partial x_{2}} \big|u \big|^{2}.
\end{align}
Integrating the last term in \eqref{parts} by parts, we have
\begin{align}
\nonumber
& 2 \Re \int_{\Omega}(x_{2}+H)\frac{\partial u}{\partial x_{2}}\nabla \cdot \bigg(\frac{1}{\varepsilon}\nabla \overline{u}\bigg) =\int_{\Omega}(x_{2}+H) 2 \Re \bigg( \frac{\partial u}{\partial x_{2}}\overline{F}\bigg)\\
\vspace{5pt}
&\qquad-2H \int_{\Gamma_{H}}\kappa^{2} \big|u\big|^{2}+\kappa^{2} \int_{\Omega} \big| u \big|^{2}.
\label{E1}
\end{align}
Using the divergence theorem, we obtain
\begin{align}
\nonumber
&\int_{\Omega}(x_{2}+H)\frac{\partial u}{\partial x_{2}}\nabla \cdot \bigg(\frac{1}{\varepsilon}\nabla \overline{u}\bigg) = -\int_{\Omega} \bigg[\frac{1}{\varepsilon}\bigg|\frac{\partial u}{\partial x_{2}} \bigg|^{2}+(x_{2}+H) \frac{1}{\varepsilon}\nabla \bigg(\frac{\partial u}{\partial x_{2}}\bigg) \cdot \overline{\nabla u} \bigg]\\
\vspace{5pt}
&+2H \int_{\Gamma_{H}}\frac{1}{\varepsilon}\bigg|\frac{\partial u}{\partial x_{2}} \bigg|^{2}+\int_{\Gamma_{R}}(x_{2}+H) \frac{1}{\varepsilon}\frac{\partial u}{\partial x_{2}}\overline{\frac{\partial u}{\partial x_{1}}}
-\int_{\Gamma_{L}}(x_{2}+H) \frac{1}{\varepsilon}\frac{\partial u}{\partial x_{2}}\overline{\frac{\partial u}{\partial x_{1}}}. 
\label{E2}
\end{align}
Using the quasi-periodicity of the solution, we see that 
\begin{equation*}
\int_{\Gamma_{R}}(x_{2}+H) \frac{1}{\varepsilon}\frac{\partial u}{\partial x_{2}}\overline{\frac{\partial u}{\partial x_{1}}}
-\int_{\Gamma_{L}}(x_{2}+H) \frac{1}{\varepsilon}\frac{\partial u}{\partial x_{2}}\overline{\frac{\partial u}{\partial x_{1}}}=0.
\end{equation*}
We take twice the real part of \eqref{E2}, use the identity $2\Re \bigg[\nabla \bigg(\frac{\partial u}{\partial x_{2}} \bigg)\cdot \overline{\nabla u} \bigg]=\frac{\partial}{\partial x_{2}}\big|\nabla u \big|^{2}$ therein, and integrate by parts. Then using quasi-periodicity again,  we obtain
\begin{equation}
\int_{\Gamma_{L}}(x_{2}+H)\frac{1}{\varepsilon}\big| \nabla u\big|^{2}-\int_{\Gamma_{R}}(x_{2}+H)\frac{1}{\varepsilon}\big| \nabla u\big|^{2}=0.
\end{equation}
Hence, \eqref{parts} can be rewritten as follows:
\begin{align}
\nonumber
2 \Re \int_{\Omega}(x_{2}+H)\frac{\partial u}{\partial x_{2}}\nabla \cdot \bigg(\frac{1}{\varepsilon}\nabla \overline{u}\bigg) &= -\int_{\Omega} \bigg[\frac{2}{\varepsilon}\bigg|\frac{\partial u}{\partial x_{2}} \bigg|^{2}-(x_{2}+H)\frac{\partial}{\partial x_{2}}\bigg(\frac{1}{\varepsilon} \bigg)\big| \nabla u\big|^{2} \bigg]\\
\vspace{5pt}
&+2H \int_{\Gamma_{H}}\bigg(\frac{2}{\varepsilon}\bigg|\frac{\partial u}{\partial x_{2}} \bigg|^{2}-\frac{1}{\varepsilon}\big|\nabla u \big|^{2} \bigg)+\int_{\Omega}\frac{1}{\varepsilon}\big|\nabla u \big|^{2}. \label{E3}
\end{align}

To complete the proof, we equate the right sides of \eqref{E1} and \eqref{E3}. The last term on the right  side of \eqref{E3} is replaced by the identity 
\begin{equation}
\label{AL-eq134}
\int_{\Omega}\frac{1}{\varepsilon}\big|\nabla u \big|^{2}=\kappa^{2}\int_{\Omega}\big|u\big|^{2}+\int_{\Gamma_{H}}\overline{u}T^{+}(u)+\int_{\Gamma_{-H}}\overline{u}T^{-}(u)-\int_{\Omega}F \overline{u},
\end{equation}
which can be obtained by setting $v=u$ in the variational problem \eqref{p1}. After rearranging some terms, the Rellich identity is shown. 

\qed

\end{document}